\documentclass[11pt,reqno]{amsart}

\usepackage{amsmath, amsthm, amssymb}
\usepackage{amsfonts}
\usepackage[ansinew]{inputenc}
\usepackage[dvips]{epsfig}
\usepackage{graphicx}
\usepackage[english]{babel}
\usepackage{comment}
\usepackage{graphicx}
\usepackage{hyperref}
\usepackage{bbm} 

\usepackage{cite}
\usepackage{graphicx}
\usepackage{amscd}
\usepackage{xcolor}
\usepackage{bm}
\usepackage{enumerate}

\usepackage{verbatim}
\usepackage{hyperref}
\usepackage{amstext}
\usepackage{latexsym}
\let\oldsqrt\sqrt
\def\sqrt{\mathpalette\DHLhksqrt}
\def\DHLhksqrt#1#2{%
\setbox0=\hbox{$#1\oldsqrt{#2\,}$}\dimen0=\ht0
\advance\dimen0-0.2\ht0
\setbox2=\hbox{\vrule height\ht0 depth -\dimen0}%
{\box0\lower0.4pt\box2}}

\allowdisplaybreaks

\newcommand{\R}{\mathbb{R}} 
\newcommand{\N}{\mathbb{N}} 

\newcommand{\osc}{\textnormal{osc}} 

\newcommand{\ov}{\overline}

\renewcommand{\phi}{\varphi}

\newcommand{\cC}{{\mathcal C}}
\newcommand{\cD}{{\mathcal D}}
\newcommand{\cE}{{\mathcal E}}

\newcommand{\cH}{{\mathcal H}}

\newcommand{\cR}{{\mathcal R}}

\newcommand{\eps}{\varepsilon}

\theoremstyle{definition}
\newtheorem{defi}{Definition}[section]
\newtheorem{remark}[defi]{Remark}

\theoremstyle{plain} 
\newtheorem{thm}[defi]{Theorem}
\newtheorem{prop}[defi]{Proposition}
\newtheorem{lemma}[defi]{Lemma}

\newtheorem{cor}[defi]{Corollary}

\theoremstyle{definition}

\numberwithin{equation}{section}

 \title[The eigenvalue problem for the regional fractional laplacian in the small order limit]{The eigenvalue problem for the regional fractional laplacian in the small order limit}

\author[Remi Yvant Temgoua, Tobias Weth]{Remi Yvant Temgoua, Tobias Weth}

\address{Goethe-Universit\"{a}t Frankfurt, Institut f\"{u}r Mathematik.
	Robert-Mayer-Str. 10, D-60629 Frankfurt, Germany, and African Institute for Mathematical Sciences in Senegal (AIMS Senegal), KM 2, Route de Joal, B.P. 1418. Mbour, S\'{e}n\'{e}gal.}

\email{temgoua@math.uni-frankfurt.de,~remi.y.temgoua@aims-senegal.org} 

\address{Goethe-Universit\"{a}t Frankfurt, Institut f\"{u}r Mathematik.
Robert-Mayer-Str. 10, D-60629 Frankfurt, Germany.}

\email{weth@math.uni-frankfurt.de}

\date{\today}

\begin{document}
\maketitle

\begin{abstract}
In this note, we study the asymptotic behavior of eigenvalues and eigenfunctions of the regional fractional Laplacian $(-\Delta)^s_{\Omega}$ as $s\rightarrow0^+.$ Our analysis leads to a study of the regional logarithmic Laplacian, which arises as a formal derivative of regional fractional Laplacians at $s= 0$. 
\end{abstract}
{\footnotesize
\begin{center}
\end{center}
}

\section{Introduction and main results}\label{introduction and main result}
In recent decades, the study of nonlocal operators has been an active area of research in different branches of mathematics. In particular, these operators are used to model problems in which different length scales are involved. In this work, we study the regional fractional Laplace operator of order $s$, which we will denote by $(-\Delta)^s_{\Omega},$ where, here and in the following, $\Omega\subset\R^N$ is a bounded open set with Lipschitz boundary. This operator is known to be the infinitesimal generator of the so-called censored stable L\'{e}vy processes and has received extensive attention in this context in recent years, see e.g. see \cite{bogdan2003censored,chen2010two,guan2005boundary,guan2006integration,guan2006reflected} and the references therein. The censored stable process is a jump process restricted to the underlying open set $\Omega$, so it only involves jumps from points in $\Omega$ to points in $\Omega$. From the point of view of partial differential equations, equations involving the regional fractional Laplacian arise as nonlocal, lower order variants of elliptic second order equations on $\Omega$ with homogeneous Neumann boundary conditions, see e.g. \cite{andreu2010nonlocal} and \cite[Theorem 1.1]{fall-ros-oton2021arxiv}.

If the underlying open set $\Omega$ equals $\R^N,$ then $(-\Delta)^s_{\Omega}$ coincides with the standard fractional Laplacian $(-\Delta)^s$. Recently, Chen and the second author \cite{chen2019dirichlet} have studied Dirichlet problems for the Logarithmic Laplacian operator $L_{\Delta}$, which arises as formal derivative $\partial_s\big|_{s=0}\,(-\Delta)^s$. In particular, they provide a relationship between the first non-zero Dirichlet eigenvalue of $(-\Delta)^s$ on $\Omega$ with that of $L_{\Delta}.$ More precisely, denoting by $\lambda^s_1(\Omega)$ resp. $\lambda^L_1(\Omega)$ the first non-zero Dirichlet eigenvalue of $(-\Delta)^s$ with corresponding $L^2$-normalized eigenfunction $u_s$ and $L_{\Delta}$ with corresponding $L^2$-normalized eigenfunction $\xi_1$, respectively, they have shown that $\lambda^L_1(\Omega)=\frac{d}{ds}|_{s=0}\lambda^s_1(\Omega)$ and $u_s\rightarrow\xi_1$ in $L^2(\Omega)$ as $s\rightarrow0^+.$
Related results for higher eigenvalues and eigenfunctions, including refined uniform regularity results and uniform convergence estimates, have been obtained more recently in \cite{feulefack2020small}. The main aim of this work is to establish analogous results in the case of the regional fractional Laplacian. As a motivation, we mention order-dependent optimization problems arising e.g. in image processing \cite{antil.bartels} and population dynamics \cite{sprekels.valdinoci, pellacci.verzini}. In many of these problems the optimal order $s$ is small. Hence the small order limit $s \to 0^+$ in $s$-dependent operator equations arises as a natural object of interest and has even been studied even in the framework of nonlinear problems recently \cite{hs.saldana}. 

To state our main results, we need to introduce some notation. Let $s\in (0,1).$  The regional fractional Laplacian $(-\Delta)^{s}_{\Omega}u$ of a function $u \in L^1(\Omega)$ is defined at a point $x\in\Omega$ by
\begin{equation}\label{definition-of-fractional-regional-operator}
(-\Delta)^{s}_{\Omega}u(x)=c_{N,s}\cD^s_{\Omega}u(x) 
\end{equation}
with
\begin{equation}\label{definition-D-operator}
\cD^s_{\Omega}u(x)= \text{P.V.}\int_{\Omega}\frac{u(x)-u(y)}{|x-y|^{N+2s}}\ dy=
\lim\limits_{\varepsilon \to 0^+}\int_{\Omega\setminus B_{\varepsilon}(x)}\frac{u(x)-u(y)}{|x-y|^{N+2s}}\ dy,
\end{equation}
provided that the limit exists. Here the normalization constant $c_{N,s}$ coincides with the one of the fractional Laplacian and is given by  
\begin{equation}\label{constant-C-n-s}
c_{N,s}:=\frac{s4^{s}\Gamma(\frac{N+2s}{2})}{\pi^{\frac{N}{2}}\Gamma(1-s)}=\frac{s(1-s)4^{s}\Gamma(\frac{N+2s}{2})}{\pi^{\frac{N}{2}}\Gamma(2-s)}.
\end{equation}
As a consequence, we have
\begin{equation}
\label{eq:def-kappa-s}
(-\Delta)^{s}_{\Omega}u(x) = (-\Delta)^s u(x) - \kappa_{\Omega,s}(x)u(x)\qquad \text{with}\qquad  
\kappa_{\Omega,s}(x)=c_{N,s}\int_{\R^N \setminus\Omega}|x-y|^{-N-2s}\ dy
\end{equation}
for $u \in L^1(\Omega)$ and $x \in \Omega$ whenever the limit in (\ref{definition-D-operator}) exists. Here we identify $u$ with its trivial extension on $\R^N$ to compute $(-\Delta)^s u(x)$.

It is important to note here that the definition of the renormalized operator $\cD^s_{\Omega}$ in (\ref{definition-D-operator}) extends to the case $s=0$. More importantly, we shall see in our first preliminary result that the family of operators $\cD^s_{\Omega}$, $s \in [0,1)$ can be expanded, in a suitable strong sense, as a convergent power series in the fractional order $s$ at $s=0$. 

\begin{thm}\label{first-main-result}
	Let $\Omega$ be a bounded open Lipschitz set in $\R^N,$ and $\alpha\in (0,1)$. Then we have 
	\begin{equation}
	\cD^s_{\Omega}u= \cD^0_{\Omega}u + \sum_{k=1}^\infty s^k \cD_k u~~ \text{for $u\in C^{\alpha}(\overline{\Omega})$  and $s \in (0,\frac{\alpha}{2})$,} \label{1-main-result-eq-n-exp}
	\end{equation}
	where, for $k \in \N$, $\cD_k u \in C(\overline{\Omega})$ is defined by 
	\begin{equation}
	\label{eq:def-D-k}
	[\cD_k u](x)  = (-1)^k2^k \int_{\Omega}\frac{u(x)-u(y)}{|x-y|^{N}} \log^k(|x-y|)\,dy.
	\end{equation}
	Here the series on the RHS of (\ref{1-main-result-eq-n-exp}) converges in $L^\infty(\Omega)$, and the convergence is uniform if $s$ is taken from a compact subset of $[0,\frac{\alpha}{2})$ and $u$ is taken from a bounded subset of $C^\alpha(\overline \Omega)$.
\end{thm}

Since
\begin{equation}
\label{eq:def-c-n}
c_{N,s}:=s c_N + o(s) \quad\text{as}~~s\to0^+,\qquad \text{with}\quad c_N:=\pi^{-\frac{N}{2}}\Gamma(\frac{N}{2}),
\end{equation}
the following is a direct corollary of Theorem~\ref{first-main-result}.

\begin{cor}
	\label{first-main-result-cor}
	Let $\Omega \subset \R^N$ be an open bounded Lipschitz set and $\alpha\in (0,1)$. For $u\in C^{\alpha}(\overline{\Omega})$, we then have 
	\begin{equation}
	(-\Delta)^{s}_{\Omega}u=s L^{\Omega}_{\Delta}u  + o(s)~~ \text{in $L^\infty(\Omega)$ as $s\rightarrow0^{+}$,} \label{1-main-result-eq}
	\end{equation}
	where
	\begin{equation}\label{m-delta}
	\bigl[L^{\Omega}_{\Delta}u\bigr](x):=c_{N}\cD^0_{\Omega}u(x)= c_N \int_{\Omega}\frac{u(x)-u(y)}{|x-y|^{N}}\ dy,~~~x\in\Omega.
	\end{equation}
	Moreover, the expansion in (\ref{1-main-result-eq}) is uniform in bounded subsets of $C^\alpha(\overline \Omega)$.
\end{cor}

In analogy to the work \cite{chen2019dirichlet}, we call $L^{\Omega}_{\Delta}=c_{N}\cD^0_{\Omega}$ the {\em regional logarithmic Laplacian} on $\Omega$. So Corollary~\ref{first-main-result-cor} states that the nonlocal operator $L^{\Omega}_{\Delta}$ arises as formal derivative $\partial_s\big|_{s=0}\,(-\Delta)^s_{\Omega}$ of regional fractional Laplacians at $s=0.$ As we shall see now in our second main result, this operator arises naturally when studying the asymptotic behavior of eigenvalues and eigenfunctions of $(-\Delta)^s_{\Omega}$ for $s$ close to $0.$ 

\begin{thm}\label{second-main-theorem-reformulation-intro}
  Let $\Omega\subset\R^N$ be a bounded open Lipschitz set, let $n \in \N$, and let $\mu^\Omega_{n,s}$ resp. $\mu^\Omega_{n,0}$ denote the $n$-th eigenvalues of the operators $(-\Delta)^s_\Omega$, $L^{\Omega}_{\Delta}$ in increasing order, respectively. Then we have
  $$
  \mu^\Omega_{n,s} \to 0\quad \text{as $s \to 0^+$}\qquad \text{and}\qquad  \frac{d}{ds}\Big|_{s=0}\,\mu ^{\Omega}_{n,s} = \lim \limits_{s \to 0^+}\frac{\mu^{\Omega}_{n,s}}{s} = \mu^{\Omega}_{n,0}.
 $$
Moreover, if, for some sequence $s_k\rightarrow0^+$, $\{\xi_{n,s_{k}}\}_{k}$ is a sequence of $L^2$-normalized eigenfunctions of $(-\Delta)^{s_k}_\Omega$ corresponding to $\mu^{\Omega}_{n,s_k},$ then $\xi_{n,s_k} \in C(\overline \Omega)$ for every $k \in \N$ and 
	\begin{equation*}
	\xi_{n,s_{k}}\rightarrow\xi_n~~~\text{uniformly in $\overline \Omega$,}
	\end{equation*}
	where $\xi_n $ is an eigenfunction of $L^{\Omega}_{\Delta}$ corresponding to $\mu_{n,0}^\Omega.$    
  \end{thm}

We stress that, here and in the following, an open bounded set $\Omega \subset \R^N$ will be called a Lipschitz set if every point $p \in \partial \Omega$ has an open neighborhood $N_p \subset \R^N$ with the property that $\partial \Omega \cap N_p$ can be written as the graph of a Lipschitz function after a suitable rotation. In the literature, this is sometimes called a strongly Lipschitz set. 

The main difficulty in the proof of Theorem~\ref{second-main-theorem-reformulation-intro} is the lack of boundedness and regularity estimates for the renormalized regional fractional Laplacian $\cD^s_\Omega$ which are {\em uniform in $s \in (0,1)$.} In fact, even for fixed $s \in (0,1)$, the elliptic boundary regularity theory for this operator has only been developed very recently with regularity estimates containing $s$-dependent constants, see \cite{audrito2020neumann,fall2020arxiv,fall-ros-oton2021arxiv}. For the proof of Theorem~\ref{second-main-theorem-reformulation-intro}, we need to consider uniform $L^\infty$-estimates related to the operator family $\cD^s_\Omega$, $s \in [0,1)$ first. In this context, we note the following result of possible independent interest. 

\begin{thm}\label{new-poisson-problem-cor-intro}
	Let $s \in [0,1)$, let $\Omega\subset\R^N$ be a bounded open Lipschitz set, let $V, f\in L^{\infty}(\Omega)$, and let $u$ be a weak solution of the problem
	\begin{equation}
	\label{bounded-poisson-problem-equation}  
	\cD^{s}_{\Omega}u + V(x)u  =  f~~~\text{in}~~~\Omega.
	\end{equation}
	Then $u \in L^\infty(\Omega)$, and there exists a constant $c_0 = c_0(N,\Omega,\|V\|_{L^\infty(\Omega)},\|f\|_{L^\infty(\Omega)}, \|u\|_{L^2(\Omega)})>0$ independent of $s$ with the property that $\|u\|_{L^\infty(\Omega)} \le c_0$ in $\Omega$.
\end{thm}

For the notion of weak solution, see Section~\ref{functional analytic}. While the uniform boundedness of the sequence $(\xi_{n,s_{k}})_k$ in Theorem~\ref{second-main-theorem-reformulation-intro} follows rather directly from Theorem~\ref{new-poisson-problem-cor-intro}, it is more difficult to see that this sequence is equicontinous on $\overline{\Omega}$. We shall prove this fact in Theorem~\ref{thm-equicontinuity} below based on a series of relative oscillation estimates and a contradiction argument. 

In view of Theorem~\ref{second-main-theorem-reformulation-intro}, it is natural to ask for upper and lower bounds for the eigenvalues of $L^{\Omega}_{\Delta}$ depending on $\Omega$. This remains an open problem. In the case of the standard fractional Laplacian, upper and lower bounds have been obtained recently in \cite{laptev.weth} by means of Fourier analysis and Faber-Krahn type estimates. We believe that different methods have to be developed to tackle the problem for the regional logarithmic Laplacian.

The article is organized as follows. In Section \ref{preliminary}, we introduce some notation and give the proof of Theorem~\ref{first-main-result}. In Section \ref{functional analytic}, we present the functional analytic framework for Poisson problem for the operator family $\cD^s_{\Omega}$ and the associated eigenvalue problem. In Section~\ref{interior regularity}, we first derive a one-sided uniform estimate for subsolutions of equations of the type $\cD^s_\Omega u + V(x)u = f$ in $\Omega$ with potential $V \in L^\infty(\Omega)$  and source function $f \in L^\infty(\Omega)$. As a corollary of this uniform estimate, we then derive Theorem~\ref{new-poisson-problem-cor-intro}. Finally, in Section~\ref{boundedness-of-eigenfunctions}, we complete the proof of Theorem~\ref{second-main-theorem-reformulation-intro}.
\\\\ 
\textbf{Acknowledgements:} This work is supported by DAAD and BMBF (Germany) within the project 57385104. The authors would like to thank Mouhamed Moustapha Fall and Sven Jarohs for valuable discussions.

\section{Preliminaries and proof of Theorem \ref{first-main-result}}\label{preliminary}
In this section, we first introduce some notation. After that, we will give the proof of Theorem~\ref{first-main-result}.

For an arbitrary subset $A \subset \R^N$, we denote by $|A|$ resp. $\chi_{A}$ the $N$-dimensional Lebesgue measure and the characteristic function of $A$, respectively. Moreover, we let $d_{A}:=\sup\{|x-y|:x,y\in A\}$ denote the diameter of $A$.  For $x\in\R^N,~r>0$, $B_r(x)$ denotes the open ball centered at $x$ with radius $r$, and $B_r := B_r(0)$. Given a function $u:A \rightarrow \R,~A\subset\R^N$, we denote by $u^+:=\max\{u,0\}$ resp. $u^-:=-\min\{u,0\}$ the positive and negative part of $u$, respectively.

Throughout the remainder of the paper, $\Omega \subset \R^N$ always denotes a bounded open Lipschitz set. For a function $u \in C^{\alpha}(\overline{\Omega})$, we put 
$$
[u]_{\alpha,x}:= \sup_{y \in \Omega} \frac{|u(x)-u(y)|}{|x-y|^\alpha} \qquad \text{for $x \in \Omega$}
$$
and
$$
[u]_\alpha:= \sup_{x \in \Omega}[u]_{\alpha,x},\qquad 
\|u\|_{C^\alpha} := \|u\|_{L^\infty(\Omega)} + [u]_\alpha.
$$

We may now give the 
\begin{proof}[Proof of Theorem~\ref{first-main-result}]
	We first note that
	\begin{equation}
	\label{simple-expansion}
	r^{-2s} = e^{-2s \ln r} = \sum_{k=0}^\infty \frac{(-2 \ln r)^k}{k!} s^k \qquad \text{for $r>0$}  
	\end{equation}
	We now fix $u \in C^\alpha(\overline \Omega)$, $x \in \Omega$ and $s \in (0,\frac{\alpha}{2})$. Moreover we define, for $s \in (0,1)$, 
	\begin{equation}
	\label{def-f}
	f: \Omega \setminus \{x\} \to \R,\qquad 
	f(y):= \frac{u(x)-u(y)}{|x-y|^{N+2s}}.
	\end{equation}
	By (\ref{simple-expansion}), we have $f(y) = \sum \limits_{k=0}^\infty s^k f_{k}(y)$ for $y \in \Omega \setminus \{x\}$ with
	\begin{equation*}
	f_{k}: \Omega \setminus \{x\} \to \R,\qquad 
	f_{k}(y):= \frac{2^k}{k!} (u(x)-u(y))  (-\ln |x-y|)^k |x-y|^{-N}.
	\end{equation*}
	Next we choose $R>0$ such that $\Omega \subset B_R(x)$ for every $x \in \Omega$, and we note that
	\begin{align*}
	&\int_{\Omega}|f_k(y)|\,dy \le \frac{2^k}{k!} [u]_\alpha \int_{\Omega} |x-y|^{\alpha-N} |\log^{k} |x-y|\Bigr|\,dy \\
	&\le \frac{2^k}{k!} [u]_\alpha \Bigl( (-1)^{k}\int_{B_1} |z|^{\alpha-N}\log^{k} |z|\,dz
	+ \int_{B_{R} \setminus B_1}|z|^{\alpha-N} \log^{k} |z|\,dz\Bigr)\\ 
	&\le \frac{2^k}{k!} [u]_\alpha |S^{N-1}| \Bigl( (-1)^{k}\int_{0}^1
	r^{\alpha-1}\log^{k}r\,dr
	+ R^{\alpha} \log^{k} R\Bigr)
	\end{align*}
	Since 
	$$
	(-1)^{k}\int_{0}^1
	r^{\alpha-1}\log^{k}r\,dr = \int_0^\infty t^k e^{-\alpha t}dt = \alpha^{-k-1}\int_0^\infty t^k e^{-t}\,dt = \frac{k!}{\alpha^{k+1}},
	$$
	we thus find that
	$$
	\int_{\Omega}|f_k(y)|\,dy \le [u]_\alpha \, c_k \qquad \text{with}\quad
	c_k = \frac{2^k}{\alpha^{k+1}} +  \frac{R^{\alpha} (2 \log R)^k}{k!}. 
	$$
	Since $\limsup \limits_{k \to \infty}\, (c_k)^{\frac{1}{k}} = \frac{2}{\alpha} < \frac{1}{s}$
	by assumption, we conclude that
	$$
	\sum_{k=j}^\infty \Bigl(\int_{\Omega}|f_k(y)|\,dy\Bigr)s^k \le [u]_\alpha 
	d_j(s) \qquad \text{with}\quad d_j(s):= \sum_{k=j}^\infty c_k s^k < \infty
	$$
	for $j \in \N$. Hence the function $g: = \sum \limits_{k=0}^\infty s^k |f_k|$
	is integrable on $\Omega$. Since
	$$
	\Bigl|\sum_{k=0}^j f_k\Bigr| \le g \qquad \text{in $\Omega \setminus \{x\}$ for every $j \in \N$,}
	$$
	it thus follows from the dominated convergence theorem that
	$$
	\cD^s_{\Omega} u(x)= \int_{\Omega} \Bigl(\sum_{k=0}^\infty s^k f_k(y)\Bigr)dy = \sum_{k=0}^\infty s^k \int_{\Omega}f_k(y)\,dy = \cD^0_\Omega u(x) + \sum_{k=1}^\infty s^k [\cD_k u](x), 
	$$
	where $[\cD_k u](x)$ is defined in (\ref{eq:def-D-k}). This holds for every $x \in \Omega$. Moreover,
	$$
	\Bigl|\cD^s_{\Omega} u(x) - \cD^0_\Omega u(x) - \sum_{k=1}^{j-1} s^k [\cD_k u](x)\Bigr| \le
	\sum_{k=j}^\infty \Bigl(\int_{\Omega}|f_k(y)|\,dy\Bigr)s^k \le [u]_\alpha \,
	d_j(s)
	$$
	for $x \in \Omega$, $u \in C^\alpha(\overline \Omega)$, where $d_j(s) \to 0$ as $j \to \infty$. Consequently, the series expansion holds in $L^\infty(\Omega)$ and uniformly for $u$ taken from a bounded subset of $C^\alpha(\overline \Omega)$.
\end{proof}

\section{Functional setting for the Poisson problem and the eigenvalue problem}\label{functional analytic}

In this section, we discuss the variational framework for the study of weak solutions to the Poisson problems
\begin{equation}\label{poisson-problem}
\cD^s_{\Omega}u = f~~~~ \text{in}~~~ \Omega,
\end{equation}
related to the operator family $\cD^s_\Omega$ for $s \in [0,1)$ and $f\in L^{2}(\Omega).$ Here and throughout this section, $\Omega \subset \R^N$ is a bounded open Lipschitz set. The variational framework for this problem is well-known for $s \in (0,1)$, and some aspects have also been studied recently in a setting related to the case $s=0$, see e.g. \cite{correa2018nonlocal}. Since we need additional properties which are not addressed in the present literature, we give a unified account for general $s \in [0,1)$ in the following.

Let us denote by $\langle\cdot,\cdot\rangle_{2}$ the usual scalar product in $L^2(\Omega)$, i.e. 
$\langle u,v\rangle_{2}=\int_{\Omega}uv\ dx$ for $u,v\in L^2(\Omega)$. We define the space $L^2_0(\Omega)$ consisting of functions $u\in L^2(\Omega)$ with zero average over $\Omega$,~i.e.
$$L^2_0(\Omega):=\Big\{u\in L^2(\Omega):\int_{\Omega}u\ dx=0\Big\}.$$
Moreover, we put
$$
\mathbb{H}^s(\Omega):=\Big\{u\in L^2(\Omega)\::\: \int_{\Omega}\int_{\Omega}\frac{(u(x)-u(y))^2}{|x-y|^{N+2s}}\ dxdy < \infty \Big \}.
$$
Then
\begin{equation}
\cE_s(u,v):=\frac{1}{2}\int_{\Omega}\int_{\Omega}\frac{(u(x)-u(y))(v(x)-v(y))}{|x-y|^{N+2s}}\ dxdy
\end{equation}
is well-defined for functions $u,v \in \mathbb{H}^s(\Omega)$. We have the following.
\begin{prop}\label{compact-embedding}
	Let $s \in [0,1)$.\\
	$(i)~\mathbb{H}^s(\Omega)$ is a Hilbert space with inner product
	$$\langle u,v\rangle_{\mathbb{H}^s(\Omega)}:=\langle u,v\rangle_{2}+\cE_s(u,v);$$
	$(ii)$	Moreover, $\mathbb{H}^s(\Omega)$ is compactly embedded into $L^2(\Omega).$
\end{prop}

Before given the proof of this Proposition, we first recall that, for $s \in (0,1)$, the space $\mathbb{H}^s(\Omega)$ coincides, by definition, with the usual fractional Sobolev space $H^s(\Omega)$. For $s \in (0,\frac{1}{2})$ this space can be identified, by trivial extension, with the space $\cH^s_0(\Omega)$ of all functions $u \in H^s(\R^N)$ with $u \equiv 0$ on $\R^N \setminus \Omega$, see e.g. \cite[Chapter 1]{grisvard2011elliptic}. This is a consequence of the fractional boundary Hardy inequality. For the case $s=0$, we have the following related property. 

\begin{lemma}
	\label{space-equivalence}
	Let $\cH(\Omega)$ be the space of all measurable functions $u:\R^N\rightarrow\R$ with $u\equiv0$ on $\R^N\setminus\Omega$ and
	$$\iint_{\substack{x,y\in\R^N\\ |x-y|\leq1}}\frac{(u(x)-u(y))^2}{|x-y|^{N}}\ dxdy<\infty,$$
	endowed with the norm
	$$
	\|u\|_{\cH(\Omega)}=\Bigl(\frac{1}{2}\iint_{\substack{x,y\in\R^N\\ |x-y|\leq1}}\frac{(u(x)-u(y))^2}{|x-y|^{N}}\ dxdy\Bigr)^{\frac{1}{2}},
	$$
	Then, by trivial extension, the space $\mathbb{H}^0(\Omega)$ is isomorphic to $\cH(\Omega)$, so there exists a constant $C>0$ with
	$$
	\frac{1}{C} \|u\|_{\cH(\Omega)} \le \|u\|_{\mathbb{H}^0(\Omega)} \le C \|u\|_{\cH(\Omega)} \qquad
	\text{for $u \in \mathbb{H}^0(\Omega)$,}
	$$
	where we identify a function $u$ on $\Omega$ with its trivial extension to $\R^N$.
\end{lemma}

We note that the space $\cH(\Omega)$ has been introduced in \cite{chen2019dirichlet} as the form domain for Dirichlet problems for the logarithmic Laplacian. 

\begin{proof}
	Let $u \in \mathbb{H}^0(\Omega)$. In the following, $C>0$ stands for a constant which may change its value from line to line but does not depend on $u$. We first note that 
	\begin{align*}
	\cE_0(u,u) &= \frac{1}{2}\int_{\Omega}\int_{\Omega}\frac{(u(x)-u(y))^2}{|x-y|^{N}}\ dxdy\\
	&\le \frac{1}{2}
	\iint_{\substack{x,y\in\R^N\\ |x-y|\leq1}}\frac{(u(x)-u(y))^2}{|x-y|^{N}}\ dxdy + \frac{1}{2}
	\iint_{\substack{x,y\in\Omega\\ |x-y|>1}}\frac{(u(x)-u(y))^2}{|x-y|^{N}}\ dxdy\\
	&\le  \|u\|_{\cH(\Omega)}^2 +  \kappa_{max} \|u\|_{L^2(\Omega)}^2\qquad \text{with}\quad 
	\kappa_{max} :=2 \max_{x \in \overline \Omega} \int_{\Omega \setminus B_1(x)}|x-y|^{-N}\,dy.
	\end{align*}
	Since $\|u\|_{L^2(\Omega)} \le C \|u\|_{\cH(\Omega)}$ for all $u \in \cH(\Omega)$ e.g. by \cite[Lemma 2.7]{FKV13}, we conclude that 
	$$
	\|u\|_{\mathbb{H}^0(\Omega)}^2 = \cE_0(u,u) + \|u\|_{L^2(\Omega)}^2 \le C \|u\|_{\cH(\Omega)}^2.
	$$
	The opposite inequality will be derived from the logarithmic boundary Hardy inequality given in \cite[Corollary 6.2.]{chen2019dirichlet}, which states that there exists a constant $C(\Omega)>0$ with the property that
	\begin{equation}
	\label{eq:log-boundary-hardy}
	\int_{\Omega}c_{\Omega}(x)u^2(x)\ dx \leq C\Big(\frac{1}{2}\int_{\Omega}\int_{\Omega}(u(x)-u(y))^2J(x-y)\ dxdy+\|u\|^2_{L^2(\Omega)}\Big)
	\qquad \text{for $u \in \cH(\Omega)$}
	\end{equation}
	with the kernel $J$ given by $J(z):=c_N\chi_{B_1}(z)|z|^{-N}$ for $z \in \R^N \setminus \{0\}$ and
	$$
	c_{\Omega}(x)=\int_{B_1(x)\setminus\Omega}|x-y|^{-N}\ dy.
	$$
	It follows from (\ref{eq:log-boundary-hardy}) that 
	\begin{align*}
	\|u\|_{\cH(\Omega)}^2 &\le \cE_0(u,u) +  \int_{\Omega}c_{\Omega}(x)u^2(x)\,dx\\  
	&\le \cE_0(u,u) + C \Big(\frac{1}{2}\int_{\Omega}\int_{\Omega}(u(x)-u(y))^2J(x-y)\ dxdy+\|u\|^2_{L^2(\Omega)}\Big)\\
	&\le C \Big(\cE_0(u,u) +\|u\|^2_{L^2(\Omega)}\Big) \le C \|u\|_{\mathbb{H}^0(\Omega)}^2. 
	\end{align*}
	The proof is thus finished.
\end{proof}

We may now complete the 

\begin{proof}[Proof of Proposition~\ref{compact-embedding}]
	The proof is well-known for $s>0$, so we restrict our attention to the case $s=0$ in the following.
	
	(i) Obviously, $\langle\cdot,\cdot\rangle_{\mathbb{H}^0(\Omega)}$ is a scalar product in $\mathbb{H}^0(\Omega).$ In the following, we prove that $\mathbb{H}^0(\Omega)$ is complete for the norm $\|\cdot\|_{\mathbb{H}^0(\Omega)}:=\sqrt{\langle\cdot,\cdot\rangle_{\mathbb{H}^0(\Omega)}}.$ Let $\{u_n\}_{n}$ be a Cauchy sequence with respect to this norm, and set
	$$v_n(x,y):=\frac{1}{\sqrt{2}}(u_n(x)-u_n(y))|x-y|^{-\frac{N}{2}}.$$
	Since $L^2(\Omega)$ is complete, $u_n \to u$ in $L^2(\Omega).$ Passing to a subsequence, we may thus assume that $u_n$ converges a.e. to $u$ on $\Omega$ and therefore $v_n$ converges a.e. on $\Omega\times\Omega$ to the function
	$$v(x,y)=(u(x)-u(y))|x-y|^{-\frac{N}{2}}.$$
	Now, by Fatou's lemma, we have that
	$$
	\int_{\Omega}\int_{\Omega}|u(x)-u(y)|^{2}|x-y|^{-N}\ dxdy\leq\liminf_{n\rightarrow\infty}\int_{\Omega}\int_{\Omega}|v_n(x,y)|^2\ dxdy= \liminf_{n\rightarrow\infty}\|u_n\|_{\mathbb{H}^0(\Omega)}^2<\infty,
	$$
	since the sequence $(u_n)_n$ is bounded in $\mathbb{H}^0(\Omega)$. Hence  $u\in\mathbb{H}^0(\Omega).$ Applying again Fatou's lemma, we find that
        \begin{align*}
        \cE_0(u_{n}-u,u_{n}-u)&=\int_{\Omega}\int_{\Omega}|v_n(x,y)-v(x,y)|^2\ dxdy \leq\liminf_{m\rightarrow\infty}\int_{\Omega}\int_{\Omega}|v_n(x,y)-v_m(x,y)|^2\ dxdy \\
	&= \liminf_{m\rightarrow\infty}\|u_n-u_m\|_{\mathbb{H}^0(\Omega)}^2 \rightarrow0~~~\text{as}~~~n\rightarrow\infty.
	\end{align*}	
	Since we have already seen that
	$u_n \to u$ in $L^2(\Omega)$, it follows that $u_n \to u$ in $\mathbb{H}^0(\Omega).$ Hence, we infer that $\mathbb{H}^0(\Omega)$ is complete and therefore is a Hilbert space.\\	
	(ii) This merely follows from the fact that, as noted in Lemma~\ref{space-equivalence}, the space $\mathbb{H}^0(\Omega)$ is isomorphic to $\cH(\Omega)$ by trivial extension, and the space $\cH(\Omega)$ is compactly embedded into $L^2(\Omega)$ by \cite[Theorem 2.1.]{correa2018nonlocal}.
	
\end{proof}
\begin{remark}\label{density-result}
	(i) The space $C^{\infty}_c(\Omega)$ is dense in $\mathbb{H}^s(\Omega)$ for $s \in (0,\frac{1}{2}]$. For $s \in (0,\frac{1}{2}]$, this is proved e.g. in \cite[Corollary 2.71.]{demengel2012functional}.
	Moreover, for $s =0$, it follows from Lemma~\ref{space-equivalence} and \cite[Theorem 3.1.]{chen2019dirichlet}.
	
	(ii) We have $C^2(\overline{\Omega}) \subset \mathbb{H}^s(\Omega)$ for $s \in [0,1)$ and  
	
	\begin{equation}\label{integration-of-M-delta}
	\int_{\Omega}[\cD^{s}_{\Omega}u]v\ dx=\cE_s(u,v)~~~\text{for all}~~~u\in C^{2}(\overline{\Omega}), v\in\mathbb{H}^s(\Omega).
	\end{equation} 
	Moreover, integrating the Poisson problem \eqref{poisson-problem} over $\Omega$ and using \eqref{integration-of-M-delta} with $v\equiv1\in C^1(\overline{\Omega})$, we see that $f \in L^2_0(\Omega)$
	is a necessary condition for the existence of a solution of  \eqref{poisson-problem}.
\end{remark}

For $s \in [0,1)$,  we consider the closed subspace 
$$
\mathbb{X}^s(\Omega):=\Big\{u\in \mathbb{H}^s(\Omega):\int_{\Omega}u\ dx=0\Big\} \:\subset \:\mathbb{H}^s(\Omega).
$$
By Proposition \ref{compact-embedding}, the embedding $\mathbb{X}^s(\Omega)\hookrightarrow L^2(\Omega)$ is compact. Furthermore, the following uniform Poincar\'{e}-type inequality holds with a constant $C_\Omega>0$:
\begin{equation}\label{poincare-inequality}
\|u\|^2_{L^2(\Omega)} \leq C_\Omega \cE_s(u,u) \qquad \text{for $s \in [0,1)$ and $u \in \mathbb{X}^s(\Omega)$.}
\end{equation}  
Indeed, for $u\in\mathbb{X}^s(\Omega)$ we have $u_{\Omega}:=\frac{1}{|\Omega|}\int_{\Omega}u\ dy=0$ and therefore,  by Jensen's inequality,
\begin{align*}
\int_{\Omega}u^2\ dx&=\int_{\Omega}|u(x)-u_{\Omega}|^2\ dx=\int_{\Omega}\Big|\frac{1}{|\Omega|}\int_{\Omega}(u(x)-u(y))\ dy\Big|^2\ dx\\&\leq\frac{1}{|\Omega|}\int_{\Omega}\int_{\Omega}(u(x)-u(y))^2\ dydx=\frac{1}{|\Omega|}\int_{\Omega}\int_{\Omega}\frac{(u(x)-u(y))^2}{|x-y|^{N+2s}}\cdot|x-y|^{N+2s}\ dydx\\&\leq C_\Omega \cE_s(u,u) \qquad \text{with}~~~C_\Omega:=2\frac{\max\{d^{N}_{\Omega},d^{N+2}_{\Omega}\}}{|\Omega|}.
\end{align*}
We note that, thanks to Proposition~\ref{compact-embedding} and \eqref{poincare-inequality}, $\mathbb{X}^s(\Omega)$ is a Hilbert space with scalar product
given by the bilinear form $(u,v) \mapsto \cE_s(u,v)$.

\begin{defi}
	\label{defi-weak-sol}
	Let $f\in L^2(\Omega)$. We say that a function $u\in\mathbb{H}^s(\Omega)$ is a weak solution of \eqref{poisson-problem} if
	\begin{equation}
	\label{eq:defi-weak-sol-eq}
	\cE_s(u,v)=\int_{\Omega}fv\ dx,~~\text{for all}~~v\in \mathbb{H}^s(\Omega).
	\end{equation}
\end{defi}

\begin{prop}
	\label{poisson-existence-weak-solution}
	For $s \in [0,1)$ and $f \in L^2_0(\Omega)$, there exists a unique weak solution $u\in\mathbb{X}^s(\Omega)$ of \eqref{poisson-problem}. 
\end{prop}

\begin{proof}
	Let $f \in L^2_0(\Omega)$. Since $\mathbb{X}^s(\Omega)$ is a Hilbert space with scalar product $\cE_s$, the Riesz representation theorem implies that there exists $u \in \mathbb{X}^s(\Omega)$ with
	$$\cE_s(u,v)=\int_{\Omega}fv\ dx,~~\text{for all}~~v\in \mathbb{X}^s(\Omega).$$
	Moreover, since $f \in L^2_0(\Omega)$, it follows that (\ref{eq:defi-weak-sol-eq}) also holds for constant functions $v \in \mathbb{H}^s(\Omega).$ Hence (\ref{eq:defi-weak-sol-eq}) holds for every $v \in \mathbb{H}^s(\Omega)$, and thus $u$ is a weak solution of 
	\eqref{poisson-problem}.   
\end{proof}

Our next aim is to study, for $s \in [0,1)$, the eigenvalue problem related to $\cD^s_{\Omega}$, that is the problem
\begin{equation}\label{eigenvalue-problem-relate-to-m}
\cD^s_{\Omega}u=\lambda u~~\text{in}~~\Omega.
\end{equation}
We consider corresponding eigenfunctions in weak sense i.e., a weak solution of \eqref{poisson-problem} with $f=\lambda u.$

\begin{prop}
	For every $s \in [0,1)$, the problem (\ref{eigenvalue-problem-relate-to-m}) admits a sequence of eigenvalues
	\begin{equation}\label{regional-eigenvalues}
	0=\lambda^{\Omega}_{0,s}<\lambda^{\Omega}_{1,s}\leq\lambda^{\Omega}_{2,s}\leq\cdots\leq\lambda^{\Omega}_{k,s}\leq\cdots\rightarrow\infty
	\end{equation}
	counted with multiplicity and a corresponding sequence of eigenfunctions which forms an orthonormal basis of $L^2(\Omega)$. Moreover, we have:
	\begin{itemize}
		\item[(i)] The eigenspace corresponding to $\lambda^{\Omega}_{0,s}=0$ is one-dimensional and consists of constant functions.
		\item[(ii)] The first non-zero eigenvalue of $\cD^s_{\Omega}$ in $\Omega$ is characterized by
		\begin{equation}\label{first-eigenvalue-of-the-regional-laplacian}
		\lambda^{\Omega}_{1,s}:=\inf\Big\{\frac{\cE_s(u,u)}{\|u\|^2_{L^2(\Omega)}}: u\in\mathbb{X}^s(\Omega) \setminus \{0\}\Big\}=\inf\Big\{\cE_s(u,u):u\in\mathbb{X}^s(\Omega),\|u\|_{L^2(\Omega)}=1\Big\}.
		\end{equation}
	\end{itemize}
\end{prop}

For $s \in (0,1)$, the proof of the characterization~(\ref{first-eigenvalue-of-the-regional-laplacian}) can be found in \cite[Theorem 3.1.]{del2015first}. For the reader's convenience, we briefly sketch a proof which covers the case $s=0$. 

\begin{proof}
	We first note that it follows in a standard way from Proposition~\ref{compact-embedding} and the nonnegativity and symmetry of the quadratic form $\cE_s$ that (\ref{eigenvalue-problem-relate-to-m}) admits a sequence of eigenvalues
	$$
	0 \le \lambda^{\Omega}_{0,s} \le \lambda^{\Omega}_{1,s}\leq\lambda^{\Omega}_{2,s}\leq\cdots\leq\lambda^{\Omega}_{k,s}\leq\cdots\rightarrow\infty
	$$
	Moreover, by definition, a function $u \in \mathbb{H}^s(\Omega)$ is an eigenfunction of (\ref{eigenvalue-problem-relate-to-m}) corresponding to the eigenvalue $\lambda=0$ if and only if $\cE_s(u,v)=0$ for every $v \in \mathbb{H}^s(\Omega)$, and this is true if and only if $u$ is constant. Hence we have $\lambda^{\Omega}_{0,s}=0$ with a one-dimensional eigenspace consisting of constant functions, and thus $\lambda^{\Omega}_{1,s}>0$. To prove~(\ref{first-eigenvalue-of-the-regional-laplacian}), we first note that 
	\begin{equation}\label{first-eigenvalue-of-the-regional-laplacian-proof}
	\lambda^{\Omega}_{1,s} \ge \inf\Big\{\cE_s(u,u):u\in\mathbb{X}^s(\Omega),\|u\|_{L^2(\Omega)}=1\Big\}
	\end{equation}
	since every eigenfunction $u$ corresponding to $\lambda^{\Omega}_{1,s}>0$ is $L^2$-orthogonal to constant functions and therefore contained in $\mathbb{X}^s(\Omega)$, whereas $\cE_s(u,u)=\lambda^{\Omega}_{1,s}$ if $\|u\|_{L^2(\Omega)}=1$.
	
	Moreover, it follows from the compactness of the embedding $\mathbb{H}^s(\Omega) \hookrightarrow L^2(\Omega)$ and the weak lower semicontinity of the functional $u \mapsto \cE_s(u,u)$ on $\mathbb{H}^s(\Omega)$ that the infimum on the RHS of (\ref{first-eigenvalue-of-the-regional-laplacian-proof}) is attained by a function $u \in \mathbb{X}^s(\Omega)$ with $\|u\|_{L^2(\Omega)}=1$.
	By Lagrange multiplier rule, we can thus find $\lambda\in\R$ such that
	$$\cE_s(u,v)=\lambda\int_{\Omega}u v\ dx~~~\text{for all}~~v\in\mathbb{X}^s(\Omega).$$
	As in the proof of Proposition~\ref{poisson-existence-weak-solution}, it then follows that $u$ weakly solves $\cD^s_\Omega u = \lambda u$, which implies that $\lambda = \lambda \|u\|_{L^2(\Omega)}^2 = \cE_s(u,u) \le 	\lambda^{\Omega}_{1,s}$ by (\ref{first-eigenvalue-of-the-regional-laplacian-proof}). Moreover, $\lambda>0$ since $u$ is non-constant. Since $\lambda^{\Omega}_{1,s}$ is the smallest positive eigenvalue by definition, it thus follows that $\lambda= \lambda^{\Omega}_{1,s}$, and hence we have equality in (\ref{first-eigenvalue-of-the-regional-laplacian-proof}).
\end{proof}

\begin{remark}
	\label{higher-eigenvalues}{\rm 
		In a standard way, it can also be shown that, for $s \in [0,1)$ the higher eigenvalues $\lambda^{\Omega}_{n,s}$, $n \in \N$ are variationally characterized as 
		\begin{equation}
		\label{eq:var-char-n}
		\lambda^{\Omega}_{n,s} = \inf_{V \in V_n^s} \sup_{u \in S_V}\cE_s(u,u).
		\end{equation}
		Here $V_n^s$ denotes the family of $n$-dimensional subspaces of $\mathbb{X}^s(\Omega)$ and $S_V:= \{u \in V\::\: \|u\|_{L^2(\Omega)} = 1\}$ for $V \in V_n^s$.
	}
\end{remark}

\section{Uniform bounds for weak subsolutions}\label{interior regularity}

In this section we establish uniform boundedness of weak solutions of problem \eqref{poisson-problem} in the case when $f\in L^{\infty}(\Omega)$. Since we are also interested in uniform bounds on $L^2$-normalized eigenfunctions of $\cD^s_\Omega$ independent of $s \in [0,1)$, it is in fact necessary to consider a generalization of \eqref{poisson-problem} involving $L^\infty$-potentials $V$.
This is the content of Theorem~\ref{new-poisson-problem-cor-intro}, which we shall derive in this section as an immediate consequence of the following more general result on subsolutions.

\begin{thm}\label{new-poisson-problem}
	Let $s \in [0,1)$, let $\Omega\subset\R^N$ be a bounded open Lipschitz set, let $V, f\in L^{\infty}(\Omega)$, and let $u\in\mathbb{H}^s(\Omega)$ be a weak subsolution of the problem 
	\begin{equation}
	\label{bounded-sub-poisson-problem-equation-0}  
	\cD^{s}_{\Omega}u +V(x)u =f~~~\text{in}~~~\Omega,
	\end{equation}
	i.e., we have
	\begin{equation}
	\label{bounded-sub-poisson-problem-equation}  
	\cE^{s}(u,\phi) + \int_{\Omega}V(x)u \phi\,dx  \le \int_{\Omega} f \phi\,dx \qquad \text{for all $\phi \in \mathbb{H}^s(\Omega)$, $\phi \ge 0$.}
	\end{equation}
	Then there exists a constant $c_0 = c_0(N,\Omega,\|V\|_{L^\infty(\Omega)},\|f\|_{L^\infty(\Omega)}, \|u^+\|_{L^2(\Omega)})>0$ independent of $s$ with the property that $u \le c_0$ in $\Omega$.
\end{thm}

As noted above, Theorem~\ref{new-poisson-problem-cor-intro} immediately follows by applying Theorem~\ref{new-poisson-problem} to $u$ and $-u$, noting that $-u$ is a weak subsolution of the equation 	\eqref{bounded-sub-poisson-problem-equation-0}  
 with $f$ replaced by $-f$. For the proof of Theorem~\ref{new-poisson-problem}, we need the following preliminary estimate.

\begin{lemma}\label{cone-property}
	Let $\Omega \subset \R^N$ be a bounded open Lipschitz set. Then there exist constants $C_0=C_0(N,\Omega)>0$ and $\delta_0= \delta_0(N,\Omega) \in (0,1)$ with the property that
	\begin{equation}
	\label{cone-estimate}
	\int_{\Omega\setminus B_{\delta}(x)}|x-y|^{-N-2s}\ dy \geq C_0\, \log \frac{\delta_0}{\delta} \qquad \text{for all}~~ \delta \in (0,\delta_0), x\in \overline{\Omega}, s\in [0,1). 
	\end{equation}
\end{lemma}

\begin{proof}
	Since the boundary $\partial\Omega$ is Lipschitz, then $\Omega$ has the uniform cone property (see for instance \cite[Theorem 1.2.2.2]{grisvard2011elliptic}).  Therefore, there exist a cone segment
	$$
	\cC_{\alpha,\delta_0}:=\{z\in\R^N: 0< |z| \le r,\: \frac{z}{|z|} \cdot e_N <\alpha\}
	$$
	for some $\delta_0 \in (0,1)$, $\alpha\in (0,\frac{\pi}{2}]$ with the property that for every $x \in \overline \Omega$ there exists a rotation $\cR_x\in O(N)$ with
	$$
	x+\cR_x(\cC_{\alpha,\delta_0})\subset\Omega.
	$$
	Setting $S_\alpha:= \{z \in S^{N-1}\::\: z \cdot e_N < \alpha\}$, we thus have  
	\begin{align*}
	&\int_{\Omega\setminus B_{\delta}(x)}|x-y|^{-N-2s} dy \geq \int_{(x+\cR_x(\cC_{\alpha,\delta_0})) \setminus B_{\delta}(x)}|x-y|^{-N-2s} dy = \int_{\cC_{\alpha,\delta_0} \setminus B_{\delta}(0)}|z|^{-N-2s} dz\\
	&\ge \int_{\cC_{\alpha,\delta_0})) \setminus B_{\delta}(0)}|z|^{-N}\ dz  \geq \cH^{N-1}(S_\alpha) \int_{\delta}^{\delta_0}\rho^{-1}d\rho = \cH^{N-1}(S_\alpha)\, \log \frac{\delta_0}{\delta},
	\end{align*}
	where $\cH^{N-1}(S_\alpha)$ is the surface measure of the set $S_\alpha \subset S^{N-1}$.
	Hence (\ref{cone-estimate}) holds with $C_0:= \cH^{N-1}(S_\alpha)$.         
\end{proof}

\begin{proof}[Proof of Theorem~\ref{new-poisson-problem}]
	In the following, we let $C_0$ and $\delta_0>0$ be given by Lemma~\ref{cone-property}. For $\delta \in (0,\delta_0)$ and $s \in [0,1)$, we consider the kernel function
	$$
	z \mapsto j_{\delta,s}(z)=\chi_{B_\delta(0)}(z) |z|^{-N-2s} 
	$$
	and the corresponding quadratic form defined by 
	\begin{equation}\label{def-quadratic form}
	\cE_s^{\delta}(v,\phi)=\frac{1}{2}\int_{\Omega}\int_{\Omega}(v(x)-v(y))(\phi(x)-\phi(y))j_{\delta,s}(x-y)\ dydx~~~\text{for}~~~v,\phi\in\mathbb{H}^s(\Omega).
	\end{equation}
	Since $u\in\mathbb{H}^s(\Omega)$ satisfies  (\ref{bounded-sub-poisson-problem-equation}), we have 
	\begin{align}\label{localization-rule}
	&\int_{\Omega} f \phi\,dx  \ge \cE_s(u,\phi) + \int_{\Omega} V(x)u(x) \phi(x)\,dx\\
	&=  \cE_{s}^\delta(u,\phi) + \int_{\Omega} V(x)u(x) \phi\,dx
	+ \frac{1}{2}\iint_{\stackrel{x,y \in \Omega}{|x-y|\ge \delta}}\frac{(u(x)-u(y))(\phi(x)-\phi(y))}{|x-y|^{N+2s}}\,dxdy \nonumber\\
	&=  \cE_{s}^\delta(u,\phi) + \int_\Omega (\gamma_{s,\delta}(x)+V(x))u(x)\phi(x)\,dx - \int_{\Omega}\kappa_{s,\delta,u}(x)\phi(x)\,dx\\
	&\ge   \cE_{s}^\delta(u,\phi) + \int_\Omega (\gamma_{s,\delta}(x)+V(x))u(x)\phi(x)\,dx - \int_{\Omega}\kappa_{s,\delta,u^+}(x)\phi(x)\,dx
	\nonumber  \end{align}
	for $\phi \in \mathbb{H}^s(\Omega)$, $\phi \ge 0$ with
	$$
	\gamma_{s,\delta}(x)=\int_{\Omega\setminus B_{\delta}(x)}|x-y|^{-N-2s}\ dy, \qquad \kappa_{s,\delta,u}(x):= \int_{\Omega \setminus B_{\delta}(x)}u(y)|x-y|^{-N-2s}\,dy   
	$$
	and
	$$
	\kappa_{s,\delta,u^+}(x):= \int_{\Omega \setminus B_{\delta}(x)}u^+(y)|x-y|^{-N-2s}\,dy.
	$$
	We note that 
	\begin{equation}
	\label{eq:inf-gamma-est}
	\inf_{x \in \Omega} \gamma_{s,\delta}(x) \ge C_0 \log \frac{\delta_0}{\delta} \qquad \text{for $\delta \in (0,\delta_0)$, $s \in [0,1)$}
	\end{equation}
	by Lemma~\ref{cone-property}. Next we fix $c>0$ and apply (\ref{localization-rule}) to $\phi_c:= (u-c)^+$, which is easily seen to be a function in $\mathbb{H}^s(\Omega)$. Since $u \phi_c \ge c \phi_c$ in $\Omega$, (\ref{localization-rule}) and (\ref{eq:inf-gamma-est}) give
	\begin{equation}\label{localization-rule-1}
	\int_{\Omega} f \phi_c\,dx \ge \cE_s^\delta(u,\phi_c)  + \Bigl(\bigl(C_0 \log \frac{\delta_0}{\delta}-\|V\|_{L^\infty(\Omega)}\bigr) c -\|\kappa_{s,\delta,u^+}\|_{L^\infty(\Omega)} \Bigr)\int_\Omega \phi_c\,dx
	\end{equation}
	where 
	$$
	\cE_s^\delta(u,\phi_c)=\cE_s^{\delta}(u-c,(u-c)^+)=\cE_s^{\delta}(\phi_c,\phi_c)-\cE_s^{\delta}((u-c)^{-},(u-c)^{+}) \geq\cE_s^{\delta}(\phi_c,\phi_c)\geq0.
	$$
	Consequently, (\ref{localization-rule-1}) implies that 
	\begin{equation}\label{localization-rule-2}
	\Bigl(\|f\|_{L^\infty(\Omega)} + \|\kappa_{s,\delta,u^+}\|_{L^\infty(\Omega)}- \bigl(C_0 \log \frac{\delta_0}{\delta}-\|V\|_{L^\infty(\Omega)}\bigr) c \Bigr) \int_\Omega \phi_c\,dx \ge 0. 
	\end{equation}
	Next, we fix $\delta \in (0,\delta_0)$ with the property that $C_0 \log \frac{\delta_0}{\delta}-\|V\|_{L^\infty(\Omega)} \ge 1$, so that (\ref{localization-rule-2}) reduces to 
	\begin{equation}\label{localization-rule-3}
	\Bigl(\|f\|_{L^\infty(\Omega)} + \|\kappa_{s,\delta,u^+}\|_{L^\infty(\Omega)}- c \Bigr) \int_\Omega \phi_c\,dx \ge 0. 
	\end{equation}
	If $c> \|f\|_{L^\infty(\Omega)} + \|\kappa_{s,\delta,u^+}\|_{L^\infty(\Omega)}$, (\ref{localization-rule-3}) implies that $\int_\Omega \phi_c\,dx= 0$ and therefore $u \le c$ in $\Omega$. We thus conclude that $u \le c_0$ with
	$$
	c_0:= \|f\|_{L^\infty(\Omega)} + \|\kappa_{s,\delta,u^+}\|_{L^\infty(\Omega)}.
	$$
	Since 
	\begin{align*}
	0 \le \kappa_{s,\delta,u^+}(x)&= \int_{\Omega \setminus B_{\delta}(x)}u^+(y)|x-y|^{-N-2s}\,dy\\
	&\le \delta^{-N-2s} \int_{\Omega}u^+(y)\,dy
	\le \delta^{-N-2}\sqrt{|\Omega|} \|u^+\|_{L^2(\Omega)}
	\end{align*}
	for $x \in \Omega$, it follows that $c_0$ only depends on $N,\Omega,\|V\|_{L^\infty(\Omega)},\|f\|_{L^\infty(\Omega)}$ and $\|u^+\|_{L^2(\Omega)}$, as claimed.
\end{proof}

\section{Uniform estimates for convergence of eigenvalues and eigenfunctions of $\cD^s_\Omega$}
\label{boundedness-of-eigenfunctions}

In this section we first prove global bounds on eigenvalues and eigenfunctions of the operator family $\cD^s_{\Omega}$. Then we shall prove convergence of eigenvalues and eigenfunctions in the limit $s \to 0^+$.

The first result of this section is the following.

\begin{prop}
	\label{uniform-bound-eigenvalues}
	Let $\Omega \subset \R^N$ be a bounded open Lipschitz set. For every $n \in \N$, $s_0 \in (0,1)$ we have
	$$
	\Lambda_{n,s_0}^\Omega:= \sup_{s \in [0,s_0]} \lambda^{\Omega}_{n,s} < \infty.
	$$
\end{prop}

\begin{proof}
	Fix $n \in \N$, $s_0 \in (0,1)$. To estimate $\lambda^{\Omega}_{n,s}$ for $s \in [0,s_0]$, we use the variational characterization~(\ref{eq:var-char-n}) and let $V$ be a fixed $n$-dimensional subspace of $C^1_{*}(\overline \Omega) = \bigl\{u \in C^1(\overline \Omega)\::\: \int_{\Omega}u\,dx = 0 \bigr\}$. For all $u\in V$, we then have
	\begin{align}
	\cE_s(u,u)&=\frac{1}{2}\int_{\Omega}\int_{\Omega}\frac{(u(x)-u(y))^2}{|x-y|^{N+2s}}\ dxdy\leq \frac{\|\nabla u\|^2_{L^{\infty}(\Omega)}}{2}\int_{\Omega}\int_{\Omega}|x-y|^{2-N-2s}\ dxdy \nonumber \\&\leq \frac{\|u\|^2_{C^1(\overline \Omega)}}{2} \int_{\Omega}\int_{B_{d_{\Omega}}(x)}|x-y|^{2-N-2s}\ dydx\leq
	\frac{\|u\|^2_{C^1(\overline \Omega)}}{4(1-s)}|\Omega|\cH^{N-1}(S^{N-1})d^{2(1-s)}_{\Omega}.   \label{uniform-bound-eigenvalues-proof-1}
	\end{align}
	Moreover, since the norms $\|\cdot\|_{C^2}$ and $\|\cdot\|_{L^2}$ are equivalent on $V$, there exists
	$C_V=C(V)>0$ such that 
	\begin{equation}
	\label{uniform-bound-eigenvalues-proof-2}
	\|u\|_{C^1(\overline \Omega)}\leq C_V\|u\|_{L^2(\Omega)} \qquad \text{for every $u \in V$.}
	\end{equation}
	Combining (\ref{uniform-bound-eigenvalues-proof-1}) and (\ref{uniform-bound-eigenvalues-proof-2}), we deduce that
	$$
	\cE_s(u,u) \le \frac{C_V }{4(1-s_0)}|\Omega|\cH^{N-1}(S^{N-1})\max \{1, d^{2}_{\Omega}\} \quad \text{for $u \in V$ with $\|u\|_{L^2(\Omega)} = 1$.}
	$$
	It thus follows from (\ref{eq:var-char-n}) that $\sup \limits_{s \in [0,s_0]} \lambda^{\Omega}_{n,s} < \infty$, as claimed.
\end{proof}

Combining Theorem~\ref{new-poisson-problem} and Proposition~\ref{uniform-bound-eigenvalues}, we obtain the following uniform bound on eigenfunctions.

\begin{thm}\label{global-bound-of-eigenfunctions-of-l-delta}
	Let $\Omega \subset \R^N$ be a bounded open Lipschitz set, let $n \in \N$, and let $s_0 \in (0,1)$. Then there exists a constant $C=C(N,\Omega,n,s_0)>0$ with the property that for every $s \in [0,s_0]$ and every eigenfunction $\xi \in \mathbb{X}^s(\Omega)$ of the eigenvalue problem (\ref{eigenvalue-problem-relate-to-m}) corresponding to the eigenvalue $\lambda^{\Omega}_{n,s}$ we have
	$$  
	\xi \in L^\infty(\Omega) \qquad \text{and}\qquad \|\xi\|_{L^{\infty}(\Omega)}\leq C \|\xi\|_{L^2(\Omega)}. 
	$$
\end{thm}

\begin{proof}
	By homogeneity, it suffices to consider eigenfunctions $\xi \in \mathbb{X}^s(\Omega)$ with $\|\xi\|_{L^2(\Omega)} = 1$. The result then follows by applying Theorem~\ref{new-poisson-problem-cor-intro} to $V \equiv -\lambda^{\Omega}_{n,s}$ and $f \equiv 0$, noting that $\|V\|_{L^\infty} = \lambda^{\Omega}_{n,s}$ is uniformly bounded independently of $s \in [0,s_0]$ by Proposition~\ref{uniform-bound-eigenvalues}.  
\end{proof}

In the remainder of this section, we study the transition from the fractional case $s>0$ to the logarithmic case $s=0$ with regard to the eigenvalues $\lambda^{\Omega}_{n,s}$ and corresponding eigenfunctions. For simplicity, we first consider the case $n=1$, that is the first positive eigenvalue $\lambda^{\Omega}_{1,s}$.

\begin{thm}\label{second-main-theorem}
	Let $\Omega\subset\R^N$ be a bounded open Lipschitz set. Then 
	\begin{equation}\label{limit-lambda-1-s}
	\lambda^{\Omega}_{1,s} \to \lambda^{\Omega}_{1,0} \qquad \text{as~~ $s \to 0^+$.}
	\end{equation}
	Moreover, if, for some sequence $s_k\rightarrow0^+$, $\{\xi_{1,s_{k}}\}_{k}$ is a sequence of $L^2$-normalized eigenfunctions of $\cD^{s_k}_{\Omega}$ corresponding to $\lambda^{\Omega}_{1,s_k},$ we have that, after passing to a subsequence,
	\begin{equation}\label{limit-xi-1-s}
	\xi_{1,s_{k}}\rightarrow\xi_1~~\text{in}~~L^2(\Omega)~ ~\text{as}~~k\rightarrow\infty,
	\end{equation}
where $\xi_1$ is an eigenfunction of $\cD^{0}_\Omega$ corresponding to $\lambda_{1,0}^\Omega.$
\end{thm}

\begin{proof}
	It is convenient to introduce the subspace $C^2_{*}(\overline{\Omega}):=\{u\in C^2(\overline{\Omega}):\int_{\Omega}u\ dx=0\}.$ Let $u\in C^2_{*}(\overline{\Omega})$ such that $\|u\|_{L^2(\Omega)}=1.$ Then Theorem~\ref{first-main-result} together with \eqref{integration-of-M-delta} yields
	\begin{equation*}
	\limsup_{s\rightarrow0^+}\lambda^{\Omega}_{1,s} \leq\limsup_{s\rightarrow0^+}\cE_s(u,u)=\lim\limits_{s\rightarrow0^+}\langle \cD^s_{\Omega}u ,u\rangle_2=\langle \cD^0_\Omega u,u\rangle_2=\cE_0(u,u).
	\end{equation*}
	Using the fact that, by Remark~\ref{density-result}, $C^2_{*}(\overline{\Omega})$ is dense in $\mathbb{X}^0(\Omega),$ we get
	\begin{equation}\label{uper-bound-of-the-derivative}
	\limsup_{s\rightarrow0^+}\lambda^{\Omega}_{1,s}\leq \inf_{\substack{u\in\mathbb{X}^0(\Omega)\\ \|u\|_{L^2(\Omega)}=1}}\cE_0(u,u)=\lambda_{1,0}^\Omega.
	\end{equation} 
	Next we consider
	$$
	\lambda_*:= \liminf \limits_{s\rightarrow0^+}\lambda^{\Omega}_{1,s} \qquad \in [0,\lambda^{\Omega}_{1,0}],
	$$
	and we let $\{s_k\}_{k\in\N}\subset (0,1)$ be a sequence with $s_k\rightarrow0^+$ as $k\rightarrow\infty$ and such that
	$\lim \limits_{k\rightarrow\infty} \lambda^{\Omega}_{1,s_k}=\lambda_*$. Moreover, we let $\xi_{1,s_k}$ be an eigenfunction associated to $\lambda^{\Omega}_{1,s_k}$ with $\|\xi_{1,s_k}\|_{L^2(\Omega)}=1$.
	We claim that
	\begin{equation}
	\label{eq:claim-upper-bound-xi-s-k}
	\limsup_{k \to \infty}  \cE_0(\xi_{1,s_k},\xi_{1,s_k}) \le   \lambda^{\Omega}_{1,0}.
	\end{equation}
	Indeed, from \eqref{uper-bound-of-the-derivative} we have,
	with
	$$
	A_\Omega:= \{(x,y) \in \Omega \times \Omega\::\: |x-y| \le 1\} \quad\text{and}\quad B_\Omega:=
	\{(x,y) \in \Omega \times \Omega\::\: |x-y| > 1\},
	$$
	the estimate
	\begin{align*}
	\lambda^{\Omega}_{1,0}+o(1)&\geq \lambda^{\Omega}_{1,s_k}=\cE_{s_k}(\xi_{1,s_k},\xi_{1,s_k})=\frac{1}{2}\int_{\Omega}\int_{\Omega}\frac{(\xi_{1,s_k}(x)-\xi_{1,s_k}(y))^2}{|x-y|^{N+2s_k}}\ dxdy\\
	&=\frac{1}{2}\Big(\iint_{A_\Omega}\frac{(\xi_{1,s_k}(x)-\xi_{1,s_k}(y))^2}{|x-y|^{N+2s_k}}\ dxdy+\iint_{B_\Omega}\frac{(\xi_{1,s_k}(x)-\xi_{1,s_k}(y))^2}{|x-y|^{N+2s_k}}\ dxdy\Big)
	\\&\geq\frac{1}{2}\Big(\iint_{A_\Omega}\frac{(\xi_{1,s_k}(x)-\xi_{1,s_k}(y))^2}{|x-y|^{N}}\ dxdy+d^{-2s_k}_{\Omega}\iint_{B_\Omega}\frac{(\xi_{1,s_k}(x)-\xi_{1,s_k}(y))^2}{|x-y|^{N}}\ dxdy\Big)\\
	&=\cE_0(\xi_{1,s_k},\xi_{1,s_k})+\frac{d^{-2s_k}_{\Omega}-1}{2}\iint_{B_\Omega}\frac{(\xi_{1,s_k}(x)-\xi_{1,s_k}(y))^2}{|x-y|^{N}}\ dxdy\\
	&\ge
	\cE_0(\xi_{1,s_k},\xi_{1,s_k})+d^{-N}_{\Omega}
	\frac{d^{-2s_k}_{\Omega}-1}{2}\iint_{B_\Omega}(\xi_{1,s_k}(x)-\xi_{1,s_k}(y))^2\ dxdy.
	\end{align*}
	If $d_\Omega \le 1$, we infer that $\cE_0(\xi_{1,s_k},\xi_{1,s_k}) \le \lambda^{\Omega}_{1,0}+o(1)$ and therefore (\ref{eq:claim-upper-bound-xi-s-k}) already follows. If $d_\Omega >1$, we estimate
	$$
	\iint_{B_\Omega}(\xi_{1,s_k}(x)-\xi_{1,s_k}(y))^2\ dxdy \le 2
	\iint_{\Omega \times \Omega}(\xi_{1,s_k}^2(x)+\xi_{1,s_k}^2(y)) dxdy \le 4 |\Omega|\|\xi_{1,s_k}\|_{L^2(\Omega)}^2 = 4 |\Omega|
	$$
	which yields
	$$
	\lambda^{\Omega}_{1,0}+o(1)\geq \cE_0(\xi_{1,s_k},\xi_{1,s_k})+2|\Omega|d^{-N}_{\Omega}(d^{-2s_k}_{\Omega}-1) = \cE_0(\xi_{1,s_k},\xi_{1,s_k})+ o(1).
	$$
	Hence (\ref{eq:claim-upper-bound-xi-s-k}) also follows in this case.
	
	As a consequence of (\ref{eq:claim-upper-bound-xi-s-k}), the sequence $\xi_{1,s_k}$ is uniformly bounded in $\mathbb{H}^0(\Omega).$ So, after passing to a subsequence, there exists $\xi_1 \in\mathbb{H}^0(\Omega)$ such that $\xi_{1,s_k}\rightharpoonup \xi_1$ in $
	\mathbb{H}^0(\Omega)$, which by Proposition~\ref{compact-embedding} implies that $\xi_{1,s_k}\rightarrow \xi_1$ in $L^2(\Omega)$.
	Consequently, $\|\xi_1\|_{L^2(\Omega)}=1$ and  $\int_{\Omega}\xi_1\ dx=0$, so in particular $\xi_1\in\mathbb{X}^0(\Omega)$.
	
	Next, from Theorem~\ref{first-main-result} and Remark~\ref{density-result}, we have that for all $\phi\in C^2_{*}(\overline{\Omega})$,
	\begin{align}\label{contradiction-argument}
	\lim\limits_{k\rightarrow\infty} \lambda^{\Omega}_{1,s_k} \langle\xi_{1,s_k},\phi\rangle_{2}=\lim\limits_{k\rightarrow\infty} \cE_{s_k}(\xi_{1,s_k},\phi)=\lim\limits_{k\rightarrow\infty}\langle \xi_{1,s_k},\cD^{s_k}_{\Omega}\phi \rangle_{2}=\langle\xi_1,\cD^{0}_{\Omega}  \phi\rangle_{2}=\cE_0(\xi_1,\phi)
	\end{align}
	Since also $\langle \xi_{1,s_k},\phi\rangle_2\rightarrow\langle \xi_1,\phi\rangle_2$ for all $\phi\in C^2_{*}(\overline{\Omega})$ as $k\rightarrow\infty$, it follows from \eqref{contradiction-argument} that 
	$$\cE_0(\xi_1,\phi)=\lambda_*\langle\xi_1,\phi\rangle_{2}~~~\text{for all}~~\phi\in C^2_{*}(\overline{\Omega}).$$
	By density, we get
	$$\cE_0(\xi_1,\phi)=\lambda_*\langle\xi_1,\phi\rangle_{2}~~~\text{for all}~~\phi\in\mathbb{X}^0(\Omega).$$
	Since $\xi_1 \in \mathbb{X}^0(\Omega) \setminus \{0\}$, we then deduce that $\lambda_* \in (0,\lambda^\Omega_{1,0}]$ is an eigenvalue of $\cD^0_{\Omega}$ with corresponding eigenfunction $\xi_1$. Since $\lambda^\Omega_{1,0}$ is the smallest positive eigenvalue of $\cD^0_{\Omega}$ by definition, we conclude that $\lambda_*= \lambda^\Omega_{1,0}$. Combining this equality with \eqref{uper-bound-of-the-derivative}, we conclude that $\lambda^{\Omega}_{1,s} \to \lambda^{\Omega}_{1,0}$ as $s \to 0^+$, as claimed in (\ref{limit-lambda-1-s}). Moreover, we have already proved above that if, for some sequence $s_k\rightarrow0^+$, $\{\xi_{1,s_{k}}\}_{k}$ is a sequence of $L^2$-normalized eigenfunctions of $\cD^{s_k}_{\Omega}$ corresponding to $\lambda^{\Omega}_{1,s_k},$ we have that
$\xi_{1,s_{k}}\rightarrow \xi_1$ in $L^2(\Omega)$ after passing to a subsequence, where $\xi_1$ is an eigenfunction of $\cD^{0}_\Omega$ corresponding to $\lambda_{1,0}^\Omega.$ The proof is thus finished.
\end{proof}

Next, we now consider the case of higher eigenvalues. We have the following.

\begin{thm}\label{main-result-for-lambda-n}
	Let $\Omega\subset\R^N$ be a bounded open Lipschitz set. Then 
\begin{equation}\label{limit-lambda-n-s-}
\lambda^{\Omega}_{n,s} \to \lambda^{\Omega}_{n,0} \qquad \text{as~~ $s \to 0^+$.}
\end{equation}
Moreover, if, for some sequence $s_k\rightarrow0^+$, $\{\xi_{n,s_{k}}\}_{k}$ is a sequence of $L^2$-normalized eigenfunctions of $\cD^{s_k}_{\Omega}$ corresponding to $\lambda^{\Omega}_{n,s_k},$ we have that, after passing to a subsequence,
\begin{equation}\label{limit-xi-n-s}
\xi_{n,s_{k}}\rightarrow\xi_n~~\text{in}~~L^2(\Omega)~ ~\text{as}~~k\rightarrow\infty,
\end{equation}
where $\xi_n$ is an eigenfunction of $\cD^{0}_\Omega$ corresponding to $\lambda_{n,0}^\Omega.$ 
\end{thm}

The proof of this theorem is similar to the one of Theorem~\ref{second-main-theorem} but somewhat more involved technically.

\begin{proof}[Proof of Theorem \ref{main-result-for-lambda-n}]
Similarly as in the proof of Theorem \ref{second-main-theorem}, we 
first show that
\begin{equation}\label{lim-sup-lambda-n}
\limsup_{s\rightarrow0^+}\lambda^{\Omega}_{n,s}\leq \lambda^{\Omega}_{n,0}.
\end{equation}
For this we consider again the subspace $C^2_{*}(\overline{\Omega}) \subset \mathbb{X}^s(\Omega)$, and we fix an $n$-dimensional subspace $V \subset C^2_{*}(\overline{\Omega})$. Then $S_V:= \{u \in V\::\: \|u\|_{L^2(\Omega)} = 1\}$ is bounded in $C^2_{*}(\overline{\Omega})$ since the $L^2$-norm and the $C^2$-norm are equivalent on $V$. Thus Theorem~\ref{first-main-result} gives, together with \eqref{integration-of-M-delta} and \eqref{eq:var-char-n}, the estimate
	\begin{align*}
          \limsup_{s\rightarrow0^+}\lambda^{\Omega}_{n,s} \leq
          \limsup_{s\rightarrow0^+} \sup_{u \in S_V}\cE_s(u,u)=\lim \limits_{s\rightarrow0^+} \sup_{u \in S_V}\langle \cD^s_{\Omega}u ,u\rangle_2&= \sup_{u \in S_V} \langle \cD^0_\Omega u,u\rangle_2\\
          &= \sup_{u \in S_V} \cE_0(u,u).
	\end{align*}
	Using again the fact that, by Remark~\ref{density-result}, $C^2_{*}(\overline{\Omega})$ is dense in $\mathbb{X}^0(\Omega)$ and that
        $$
		\lambda^{\Omega}_{n,0} = \inf_{V \in V_n^0} \sup_{u \in S_V}\cE_0(u,u),
                $$
                by \eqref{eq:var-char-n}, where $V_n^0$ denotes the family of $n$-dimensional subspaces of $\mathbb{X}^0(\Omega)$, we deduce (\ref{lim-sup-lambda-n}).
                
Next we show the corresponding liminf inequality. For this, we fix $n \in \N$ and set 
\begin{equation*}
  \lambda^*_{j}:=\liminf_{s\to 0^+}\lambda^{\Omega}_{j,s}   \qquad \text{for $j=1,\dots,n$,}
\end{equation*}
noting that
\begin{equation}
  \label{lim-sup-lambda-n-1}
  \lambda^*_{j} \le \lambda^*_n \qquad \text{for $j=1,\dots,n$}
\end{equation}
since the sequence of numbers $\lambda^{\Omega}_{j,s}$ is increasing in $j$ for every $s \in (0,1)$. Moreover, we choose a sequence of numbers $s_k \in (0,1)$, $k \in \N$ with $s_k \to 0^+$ and $\lambda^{\Omega}_{n,s_k} \to \lambda^{*}_{n}$ as $k \to \infty$. We then choose, for every $k \in \N$, a system of $L^2$-orthonormal eigenfunctions $\xi_{1,s_k},\dots,\xi_{n,s_k}$   associated to the eigenvalues $\lambda^{\Omega}_{1,s_k},\dots,\lambda^{\Omega}_{n,s_k}$.

Proceeding precisely as in the proof of Theorem \ref{second-main-theorem}, we find that $\xi_{j,s_k}$ is uniformly bounded in $\mathbb{H}^0(\Omega)$ for $j=1,\dots,n$. Therefore, after passing to a subsequence, there exists $\xi_j \in\mathbb{H}^0(\Omega)$ such that $\xi_{j,s_k}\rightharpoonup \xi_j$ in $\mathbb{H}^0(\Omega)$ for $j =1,\dots,n$, which by Proposition \ref{compact-embedding} implies that $\xi_{j,s_k}\rightarrow \xi_j$ in $L^2(\Omega)$ for $j=1,\dots,n$.

The $L^2$-convergence implies that the functions $\xi_1,\dots,\xi_n$ are also $L^2$-orthonormal. Moreover, for $j=1,\cdots,n,$ we have, by Theorem \ref{first-main-result} and Remark \ref{density-result},  
\begin{align}\nonumber
  \lambda^*_{j} \langle \xi_j,\phi\rangle_{2} &=  \lim\limits_{k\rightarrow\infty}\lambda^{\Omega}_{j,s_k}\langle\xi_{j,s_k},\phi\rangle_{2}=\lim\limits_{k\rightarrow\infty}\cE_{s_k}(\xi_{j,s_k},\phi)\\
                                              &=\lim\limits_{k\rightarrow\infty}\langle\xi_{j,s_k},\cD^{s_k}_{\Omega}\phi\rangle_{2}=\langle \xi_j,\cD^0_{\Omega}\phi\rangle_{2}=\cE_0(\xi_j,\phi)
\qquad \text{for $\phi\in C^2_*(\overline{\Omega})$.}
                                                \label{toward-liminf}
\end{align}
By density of $C^2_{*}(\overline{\Omega})$ in $\mathbb{X}^0(\Omega)$, we thus have 
$$	\cE_0(\xi_j,\phi)=\lambda^{*}_j\langle \xi_j,\phi\rangle_{2}~~\text{for all}~~\phi\in\mathbb{X}^0(\Omega), \: j=1,\dots,n.
$$
Therefore, $\lambda^{*}_j$ is an eigenvalue of $\cD^0_{\Omega}$ with corresponding eigenfunction $\xi_j$ for $j=1,\dots,n.$ Now, by considering in particular the $n$-dimensional subspace $V=\text{span}\{\xi_1,\xi_2,\dots,\xi_n\}$ of $\mathbb{X}^0(\Omega)$ in \eqref{eq:var-char-n}, it follows that
\begin{equation}\label{uper-bound-of-lambda-k}
\lambda^{\Omega}_{n,0}\leq\sup_{u\in S_V}\cE_0(u,u).
\end{equation}
Moreover, every $u\in S_V$ writes as $u=\sum \limits_{j=1}^{n}c_j\xi_j$ with $c_j \in \R$ satisfying $\sum \limits_{j=1}^{n}c^2_j=1$, so we have 
$$
\cE_0(u,u)=\cE_0\Big(\sum_{j=1}^{n}c_j\xi_j,\sum_{j=1}^{n}c_j\xi_j\Big)=\sum_{i,j=1}^{n}c_ic_j\lambda^{*}_{i}\langle \xi_i,\xi_j\rangle_{2} =\sum_{i=1}^{n}c^2_i\lambda^{*}_{i}\leq \lambda^{*}_{n}\sum_{i=1}^{n}c^2_i = \lambda^{*}_{n}
$$
by \eqref{lim-sup-lambda-n-1}. Hence \eqref{uper-bound-of-lambda-k} yields that 
\begin{equation}\label{limit-inf-lambda-n}
\lambda^{\Omega}_{n,0}\leq\lambda^{*}_{n} = \liminf_{s\rightarrow0^+}\lambda^{\Omega}_{n,s}
\end{equation}
Combining \eqref{lim-sup-lambda-n} and \eqref{limit-inf-lambda-n} now shows that $\lambda^{\Omega}_{n,s}\to\lambda^{\Omega}_{n,0}$ as $s\to0^+$, as claimed in \eqref{limit-lambda-n-s-}. The rest of the proof follows exactly as in the case of Theorem \ref{second-main-theorem}.
\end{proof}

Next, we wish to study the uniform convergence of sequences of eigenfunctions of $\cD^{s_k}_{\Omega}$ associated with a sequence $s_k \to 0^+$. We first state a uniform equicontinuity result in a somewhat more general setting.

\begin{thm}
\label{thm-equicontinuity}  
Let $\Omega \subset \R^N$ be a bounded Lipschitz set. Moreover, let $(s_k)_k$ be a sequence in $(0,1)$ with $s_k \to 0^+$, and let $\phi_k \in C(\overline \Omega)$, $k \in \N$ be functions with 
\begin{equation}
\label{bound-diff-operator}  
\|\phi_k\|_{L^\infty(\Omega)} \le C \qquad \text{and}\qquad   \Bigl|  \int_{\Omega} \frac{\phi_{k}(x)-\phi_{k}(y)}{|x-y|^{N+2s_k}}dy\Bigr| \le C \quad \text{for all $x \in \overline{\Omega}$, $k \in  \N$}
\end{equation}
with a constant $C>0$. Then the sequence $(\phi_{k})_k$ is equicontinuous.  
\end{thm}

\begin{proof}
Since $s_k \to 0^+$, we may assume, without loss of generality, that $s_k \in (0,\frac{1}{4})$ for every $k \in \N$. Moreover, relabeling the functions $\phi_k$ if necessary, we may assume that the sequence $s_k$ is monotone decreasing.  Arguing by contradiction, we assume that there exists a point $x_0 \in \overline{\Omega}$ such that the sequence $(\phi_k)_k$ is not equicontinuous at $x_0$, which means that 
  \begin{equation}
    \label{eq:positive-osc-limit}
\lim_{t \to 0^+} \sup_{k \in \N} \:\underset{B_t(x_0)\cap \ov\Omega}{\osc}\: \phi_k= \eps>0.
  \end{equation}
  This limit exists since the function
$$
(0,\infty) \to [0,\infty), \qquad t \mapsto \sup_{k \in \N} \:\underset{B_t(x_0)\cap \ov\Omega}{\osc}\: \phi_k
$$
is bounded by assumption and nondecreasing. Without loss of generality, to simplify the notation, we may assume that $x_0 = 0 \in \overline{\Omega}$.
We first choose $\delta>0$ sufficiently small so that 
\begin{equation}
  \label{eq:delta-small-condition}
\frac{\eps-\delta}{2^{N+2}}-2 \cdot 3^{N} \delta > 0.
\end{equation}
We then choose $t_0 \in (0,1)$ sufficiently small so that
\begin{equation}
    \label{eq:positive-osc-limit-cons}
\eps \le \sup_{k \in \N} \:\underset{B_t\cap \ov\Omega}{\osc}\: \phi_k  \le \eps +\delta  \qquad \text{for $0 < t \le {2t_0}$.}
\end{equation}
From (\ref{bound-diff-operator}) and the assumption that the sequence $(\phi_k)_k$ is uniformly bounded in $\overline{\Omega}$, it follows that there exists a constant $C_1=C_1(t_0)>0$ with
\begin{equation}
\label{bound-diff-operator-1}  
  \Bigl|  \int_{B_{t_0}(x)\cap \Omega} \frac{\phi_k(x)-\phi_k(y)}{|x-y|^{N+2s_k}}dy\Bigr| \le C_1 \qquad \text{for all $x \in \overline \Omega$, $k \in  \N$.}
\end{equation}
Next, we choose a sequence of numbers $t_k \in (0,\frac{t_0}{5})$ with $t_k \to 0^+$ and
\begin{equation}
  \label{eq:t_n-power-s-n-est}
C_2:= \inf_{k \in \N} t_k^{s_k} >0.
\end{equation}
We then define a strictly increasing sequence of numbers $\sigma_k$, $k \in \N$ inductively with the property that
\begin{equation}
  \label{eq:eps-half-lower-bound-preliminary}
 \underset{B_{t_k}\cap \ov\Omega}{\osc}\, \phi_{\sigma_k} \ge \eps-\delta  \qquad \text{for all $k \in \N$.}
\end{equation}
For this, we first note that (\ref{eq:positive-osc-limit-cons}) implies that
there exists some $\sigma_1 \in \N$ with
$$
\underset{B_{t_1} \cap \ov\Omega}{\osc}\, \phi_{\sigma_1} \ge \eps-\delta.
$$
Next, suppose that $\sigma_1< \dots < \sigma_k$ are already defined for some $k \in \N$. Since the finite set of functions $\{\phi_{\sigma_1},\dots,\phi_{\sigma_k}\}$ is equicontinuous on $\overline \Omega$ by assumption, there exists $t' \in (0,t_{k+1})$ with the property that
$$
 \underset{B_{t'}\cap \ov\Omega}{\osc}\, \phi_{\ell} < \eps-\delta  \qquad \text{for $\ell = \sigma_1,\dots,\sigma_k$.}
$$
Hence, by (\ref{eq:positive-osc-limit-cons}), there exists some $\sigma_{k+1} \in \N$, $\sigma_{k+1}>\sigma_k$ with 
$$
\eps-\delta \le \underset{B_{t'}\cap \ov\Omega}{\osc}\, \phi_{\sigma_{k+1}} \le
\underset{B_{t_{k+1}}\cap \ov\Omega}{\osc }\, \phi_{\sigma_{k+1}}.
$$
With this inductive choice, (\ref{eq:eps-half-lower-bound-preliminary}) holds for all $k \in \N$. Moreover, since $\sigma_k \ge k$ and therefore $s_{\sigma_k} \le s_k$, we have $t_k^{s_{\sigma_k}} \ge t_k^{s_k} \ge C_2$ for every $k \in \N$ by (\ref{eq:t_n-power-s-n-est}) and since $t_k \in (0,1)$. Hence we may pass of a subsequence, replacing $s_k$ by $s_{\sigma_k}$ and $\phi_k$ by $\phi_{\sigma_k}$ in the following, with the property that (\ref{eq:t_n-power-s-n-est}) still holds and  
\begin{equation}
  \label{eq:eps-half-lower-bound}
\eps-\delta \le \underset{B_{t_k}\cap \ov\Omega}{\osc}\, \phi_k \le \eps +\delta  \qquad \text{for all $k \in \N$.}
\end{equation}
By \eqref{eq:eps-half-lower-bound}, we may write 
    \begin{equation}
      \label{eq:introduce-d-n}
\phi_k(\overline{B_{t_k}\cap \Omega}) = [d_k-r_k,d_k+r_k] \qquad \text{for $k \in \N$ with some $d_k \in \R$, $r_k \ge \frac{\eps-\delta}{2}$.}
    \end{equation}
Together with \eqref{eq:positive-osc-limit-cons} and the fact that $\overline{B_{t_k} \cap \Omega} \subset \overline{B_{{2t_0}} \cap \Omega}$, we deduce that 
    \begin{equation}
      \label{eq:B-t-0-interval-est}
\phi_k(\overline{B_{{2t_0}}\cap \Omega}) \subset [d_k-\frac{\eps+3\delta}{2}\:,\:d_k+\frac{\eps+3\delta}{2}].
    \end{equation}
Moreover, we let
$$
c_k:= \int_{\Omega \cap (B_{t_0} \setminus B_{3t_k})}|y|^{-N-2s_k}\,dy \qquad \text{for $k \in \N$,}
$$
and we note that
\begin{equation}
  \label{eq:c-n-infty}
c_k \to \infty\qquad \text{as $k \to \infty$}
\end{equation}
by Lemma~\ref{cone-property}. We now set
$$
A_+^k : = \{y \in \Omega \cap (B_{{t_0}}\setminus B_{3 t_k})\::\: \phi_k(y) \ge d_k\}\quad \text{and}\quad 
A_-^k : = \{y \in \Omega \cap (B_{{t_0}}\setminus B_{3 t_k})\::\: \phi_k(y) \le d_k\}.
$$
Since
$$
c_k \le \int_{A_+^k} |y|^{-N-2s_k}\,dy + \int_{A_-^k} |y|^{-N-2s_k}\,dy \qquad \text{for all $k \in \N$,}
$$
we may again pass to a subsequence such that
$$
\int_{A_+^k} |y|^{-N-2s_k}\,dy \ge \frac{c_k}{2} \quad \text{for all $k \in \N$}\qquad \text{or}\qquad
\int_{A_-^k} |y|^{-N-2s_k}\,dy \ge \frac{c_k}{2} \quad \text{for all $k \in \N$.}
$$
Without loss of generality, we may assume that the second case holds (otherwise we may replace $\phi_k$ by $-\phi_k$ and $d_k$ by $-d_k$). 
We then define the Lipschitz function $\psi_k \in C_c(\R^N)$ by
$$
\psi_k (x) = \left \{
  \begin{aligned}
    &2 \delta ,&&\qquad |x| \le t_k\\
    &0,&& \qquad |x| \ge 2 t_k\\
    &\frac{2 \delta}{t_k}(2 t_k - |x|),&& \qquad t_k \le |x| \le 2 t_k. 
  \end{aligned}
\right.
$$
We also define, for $k \in \N$, 
$$
\tau_k: \overline{\Omega} \to \R, \qquad \tau_k(x)= \phi_k(x) + \psi_k(x)
$$
By \eqref{eq:B-t-0-interval-est}, we have 
$$
\tau_k = \phi_k \le d_k + \frac{\eps+3\delta}{2}\le d_k + r_k +2 \delta \qquad \text{in $\overline{\Omega \cap (B_{{2t_0}} \setminus B_{2t_k})}$.}
$$
Moreover, since $d_k+r_k \in \phi_k(\overline{B_{t_k} \cap \Omega})$ by \eqref{eq:introduce-d-n}, we have 
$$
d_k+r_k + 2 \delta  \in \tau_k(\overline{B_{t_k} \cap \Omega}) \subset \tau_k(B_{2t_k} \cap \overline{\Omega}).
$$
Consequently, $\max \limits_{\overline{B_{2t_0}\cap \Omega}}\, \tau_k$ is attained at a point $x_k \in B_{2t_k} \cap \overline{\Omega}$  with
$$
\tau_k(x_k) \ge d_k+r_k + 2 \delta 
$$
which implies that 
\begin{equation}
  \label{eq:phi-n-x-n-lower-bound}
\phi_k(x_k) \ge d_k +r_k \ge d_k + \frac{\eps-\delta}{2}.
\end{equation}
By \eqref{bound-diff-operator-1} and since $B_{3t_k} \cap \Omega \subset B_{t_0}(x_k) \cap \Omega$ for $k \in \N$ by construction, we have that
\begin{align}
  C_1 &\ge  \int_{B_{t_0}(x_k)\cap \Omega} \frac{\phi_k(x_k)-\phi_k(y)}{|x_k-y|^{N+2s_k}}dy \nonumber\\
      &= \int_{B_{3 t_k}\cap \Omega} \frac{\phi_k(x_k)-\phi_k(y)}{|x_k-y|^{N+2s_k}}dy +
      \int_{\Omega \cap  (B_{t_0}(x_k) \setminus B_{3 t_k})} \frac{\phi_k(x_k)-\phi_k(y)}{|x_k-y|^{N+2s_k}}dy. \label{C-1-first-est}
\end{align}
To estimate the first integral, we note that, by definition of the function $\psi_k$,
$$
|\psi_k(x)-\psi_k(y)| \le \frac{2 \delta}{t_k}|x-y| \qquad \text{for all $x,z \in \R^N$.}
$$
Moreover, by the choice of $x_k$ we have $\tau_k(x_k) \ge \tau_k(y)$ for all $y \in B_{3t_k} \cap \Omega$. Consequently,
\begin{align}
  &\int_{B_{3 t_k}\cap \Omega} \frac{\phi_k(x_k)-\phi_k(y)}{|x_k-y|^{N+2s_k}}dy= \int_{B_{3 t_k}\cap \Omega} \frac{\tau_k(x_k)-\tau_k(y)}{|x_k-y|^{N+2s_k}}dy-
    \int_{B_{3 t_k}\cap \Omega} \frac{\psi_k(x_k)-\psi_k(y)}{|x_k-y|^{N+2s_k}}dy \nonumber\\
    &\ge - \int_{B_{3 t_k}\cap \Omega} \frac{\psi(x_k)-\psi(y)}{|x_k-y|^{N+2s_k}}dy
      \ge - \frac{2 \delta}{t_k}\int_{B_{3 t_k}} |x_k-y|^{1-N-2s_k}dy \ge - \frac{2 \delta}{t_k}\int_{B_{3 t_k}} |y|^{1-N-2s_k}dy
      \nonumber\\
  &= - \frac{3^{1-2s_k} \omega_{N-1} 2 \delta t_k^{-2s_k}}{1-2s_k} \ge  - 12 \omega_{N-1} \delta t_k^{-2s_k} \ge -C_3
 \label{C-3-est}   
\end{align}
with a constant $C_3>0$ independent of $k$. Here we used \eqref{eq:t_n-power-s-n-est} and the standard estimate 
        \begin{equation*}
       \int_{B_t}|x-z|^{\rho-N}\,dz \le \int_{B_t}|z|^{\rho-N}\,dz = \frac{\omega_{N-1}t^\rho }{\rho} \qquad \text{for every $t>0$, $\rho \in (0,N)$ and $x \in \R^N$.}
        \end{equation*}
To estimate the second integral in \eqref{C-1-first-est} we first note, since $x_k \in B_{2t_k}$, we have that 
$$
2|y| \ge  |y-x_k| \ge \frac{|y|}{3} \qquad \text{for every $k \in \N$ and $y \in \R^N \setminus B_{3 t_k}$.}
$$
Moreover, by \eqref{eq:positive-osc-limit-cons}, \eqref{eq:B-t-0-interval-est}, and \eqref{eq:phi-n-x-n-lower-bound} we have 
$$
\eps +\delta \ge \phi_k(x_k)-\phi_k(y) \ge d_k + \frac{\eps-\delta}{2} - \phi_k(y) \ge -2 \delta 
$$
for $y \in B_{t_0}(x_k)\cap \Omega \subset B_{2t_0} \cap \Omega$. Consequently, combining \eqref{C-1-first-est} and \eqref{C-3-est}, using again \eqref{eq:phi-n-x-n-lower-bound} and the fact that $x_k \in B_{2t_k}$, we may estimate as follows:
\begin{align*}
&C_1 + C_3\ge  \int_{(B_{t_0}(x_k) \setminus B_{3 t_k})\cap \Omega} \frac{\phi_k(x_k)-\phi_k(y)}{|y-x_k|^{N+2s_k}}dy\\
  &\ge   \int_{(B_{t_0}(x_k) \setminus B_{3 t_k})\cap \Omega} \frac{[\phi_k(x_k)-\phi_k]^+(y)}{|y-x_k|^{N+2s_k}}dy - 2 \delta \int_{(B_{t_0}(x_k) \setminus B_{3 t_k})\cap \Omega}\!\! |y-x_k|^{-N-2s_k}dy  \\
  &\ge   \frac{1}{2^{N+2s_k}}\!\int_{(B_{t_0}(x_k) \setminus B_{3 t_k})\cap \Omega}\!\! \frac{[\phi_k(x_k)-\phi_k]^+(y)}{|x_k-y|^{N+2s_k}}dy - 2 \!\cdot \! 3^{N+2s_k} \delta \!\! \int_{(B_{t_0}(x_k) \setminus B_{3 t_k})\cap \Omega} |y|^{-N-2s_k}dy  \\
           &\ge   \frac{1}{2^{N+2s_k}}\Bigl(\int_{(B_{t_0} \setminus B_{3 t_k}) \cap \Omega} \frac{[\phi_k(x_k)-\phi_k]^+(y)}{|y|^{N+2s_k}}dy - \int_{(B_{t_0} \setminus B_{t_0}(x_k))\cap \Omega} \frac{[\phi_k(x_k)-\phi_k]^+(y)}{|y|^{N+2s_k}}dy\Bigr)\\
 &\qquad- 2 \,\cdot \,3^{N+2s_k} \delta \,\Bigl( \int_{(B_{t_0} \setminus B_{3 t_k})\cap \Omega} |y|^{-N-2s_k}dy +\int_{(B_{t_0}(x_k) \setminus B_{t_0})} |y|^{-N-2s_k}dy \Bigr)  \\
           &\ge   \frac{1}{2^{N+2s_k}}\Bigl(r_k \int_{A_k^-}|y|^{-N-2s_k} dy - (\eps+\delta)\int_{B_{t_0} \setminus B_{t_0}(x_k)} |y|^{-N-2s_k}dy\Bigr) \\
&\qquad- 2 \cdot 3^{N+2s_k} \delta \Bigl(c_k + \int_{B_{t_0}(x_k) \setminus B_{t_0}} |y|^{-N-2s_k}dy\Bigr)\\
           &\ge   \Bigl(\frac{r_k}{2 \cdot 2^{N+2s_k}}-2 \cdot 3^{N+2s_k} \delta\Bigr)c_k\\
 &\qquad  - \frac{(\eps+\delta)}{2^{N+2s_k}}\!\int_{B_{t_0} \setminus B_{t_0-2t_k}} |y|^{-N-2s_k}dy\:-\: 2\cdot3^{N+2s_k} \delta \!\int_{B_{t_0+2t_k} \setminus B_{t_0}} |y|^{-N-2s_k}dy\\
           &\ge   \Bigl(\frac{\eps-\delta}{2^{N+2+2s_k}}-2 \cdot 3^{N+2s_k} \delta\Bigr)c_k - o(1) = \Bigl(\frac{\eps-\delta}{2^{N+2}}-2 \cdot 3^{N} \delta \,+\, o(1)\Bigr)c_k - o(1)
\end{align*}
as $k \to \infty$, where we used \eqref{eq:introduce-d-n}. By our choice of $\delta>0$ satisfying \eqref{eq:delta-small-condition}, we arrive at a contradiction to \eqref{eq:c-n-infty}. 
The proof is thus finished.
\end{proof}

Finally, we complete the

\begin{proof}[Proof of Theorem~\ref{second-main-theorem-reformulation-intro}]
  Since $c_{N,s}:=s c_N + o(s)$ as $s\to0^+$ with $c_N=\pi^{-\frac{N}{2}}\Gamma(\frac{N}{2})$ and $L^{\Omega}_{\Delta}=c_{N}\cD^0_{\Omega}$, then the first part of Theorem \ref{second-main-theorem-reformulation-intro} is just a reformulation of  Theorems~\ref{second-main-theorem} and \ref{main-result-for-lambda-n}.

  To see the second part, we first note that $\xi_{n,s_k} \in C(\overline \Omega)$ for every $k \in \N$ by \cite[Theorem 1.3, see also Theorem 4.7]{audrito2020neumann}. We may then apply Theorem~\ref{thm-equicontinuity} to the sequence $(\xi_{n,s_k})_k$ in place of $(\phi_k)_k$, noting that assumption~(\ref{bound-diff-operator}) is satisfied by Theorems~\ref{uniform-bound-eigenvalues} and \ref{global-bound-of-eigenfunctions-of-l-delta}. Consequently, the sequence $(\xi_{n,s_k})_k$ is both bounded in $C(\overline \Omega)$ and equicontinuous on $\overline \Omega$, so it is relatively compact in $C(\overline \Omega)$ by the Arzel\`a-Ascoli Theorem. Combining this fact with the convergence property $\xi_{n,s_k} \to \xi_n$ in $L^2(\Omega)$ stated in Theorem~\ref{main-result-for-lambda-n}, it follows that
$\xi_{n,s_k} \to \xi_n$ in $C(\overline \Omega)$.  
\end{proof}


\begin{thebibliography}{10}




\bibitem{andreu2010nonlocal} F. Andreu-Vaillo, J. M. Maz\'{o}n, J. D. Rossi, and  J. J. Toledo-Melero, \emph{Nonlocal diffusion problems.} No. 165. American Mathematical Soc., 2010.

\bibitem{antil.bartels} H. Antil, S. Bartels, 
\emph{Spectral approximation of fractional PDEs in image processing and phase field modeling.} Comput. Methods Appl. Math. 17 (2017): 661--678. 
  
\bibitem{audrito2020neumann} A. Audrito, J.-C. Felipe-Navarro, and X. Ros-Oton, \emph{The Neumann problem for the fractional Laplacian: regularity up to the boundary.} arXiv preprint arXiv:2006.10026 (2020).

\bibitem{bogdan2003censored} K. Bogdan, K. Burdzy, and Z.-Q. Chen, \emph{Censored stable processes.} Probability theory and related fields 127.1 (2003): 89-152.

\bibitem{chen2010two} Z.-Q. Chen, P. Kim, and R. Song, \emph{Two-sided heat kernel estimates for censored stable-like processes.} Probability theory and related fields 146.3-4 (2010): 361.

\bibitem{chen2019dirichlet} H. Chen, and T. Weth, \emph{The Dirichlet problem for the logarithmic Laplacian.} Communications in Partial Differential Equations 44.11 (2019): 1100-1139.

\bibitem{correa2018nonlocal} E. Correa, and A. De Pablo, \emph{Nonlocal operators of order near zero.} Journal of Mathematical Analysis and Applications 461.1 (2018): 837-867.

\bibitem{del2015first} L. M. Del Pezzo, and A. M. Salort, \emph{The first non-zero Neumann p-fractional eigenvalue.} Nonlinear Analysis: Theory, Methods \& Applications 118 (2015): 130-143.

\bibitem{demengel2012functional} F. Demengel, G. Demengel, and R. Ern\'{e}, \emph{Functional spaces for the theory of elliptic partial differential equations.} London: Springer, 2012.

\bibitem{fall2020arxiv} M. M. Fall, \emph{Regional fractional Laplacians: Boundary regularity.} arXiv preprint arXiv:2007.04808 (2020).

\bibitem{fall-ros-oton2021arxiv} M. M. Fall, and X. Ros-Oton, \emph{Global Schauder theory for minimizers of the $H^s(\Omega)$-energy.} arXiv preprint arXiv:2106.07593 (2021).

\bibitem{FKV13} M.~Felsinger, M.~Kassmann, and P.~Voigt, \emph{The Dirichlet problem for nonlocal operators.} Math. Zeit. (2015): 279 779--809.

  
\bibitem{feulefack2020small} P. A. Feulefack, S. Jarohs, and T. Weth, \emph{Small order asymptotics of the Dirichlet eigenvalue problem for the fractional Laplacian.} arXiv preprint arXiv:2010.10448 (2020).

\bibitem{grisvard2011elliptic} P. Grisvard, \emph{Elliptic problems in nonsmooth domains. } Society for Industrial and Applied Mathematics, 2011.

\bibitem{guan2006integration} Q.-Y. Guan, \emph{Integration by parts formula for regional fractional Laplacian.} Communications in mathematical physics 266.2 (2006): 289-329.

\bibitem{guan2005boundary} Q.-Y. Guan, and Z.-M. Ma, \emph{Boundary problems for fractional Laplacians.} Stochastics and Dynamics 5.03 (2005): 385-424.

\bibitem{guan2006reflected} Q.-Y. Guan, and Z.-M. Ma, \emph{Reflected symmetric $\alpha$-stable processes and regional fractional Laplacian.} Probability theory and related fields 134.4 (2006): 649-694.

\bibitem{laptev.weth} A. Laptev, T. Weth, \emph{Spectral properties of the logarithmic Laplacian.}
Anal. Math. Phys. 11 (2021), Paper No. 133, 24 pp. 
  
\bibitem{hs.saldana} V. Hern\'andez-Santamaría, A. Salda\~{n}a, \emph{Small order asymptotics for nonlinear fractional problems.}
arXiv preprint arXiv:2108.00448

\bibitem{pellacci.verzini} B. Pellacci, G. Verzini, \emph{Best dispersal strategies in spatially heterogeneous environments: optimization of the principal eigenvalue for indefinite fractional Neumann problems.} J. Math. Biol. 76 (2018): 1357--1386. 
 
\bibitem{sprekels.valdinoci} J. Sprekels, and E. Valdinoci, \emph{A new type of identification problems: optimizing the fractional order in a nonlocal evolution equation.} SIAM J. Control Optim. 55 (2017): 70--93. 



  
\end{thebibliography}
\bibliographystyle{ieeetr}

\end{document}